\documentclass[12pt,reqno,draft]{article} 
\usepackage{amsmath,amssymb,amsthm,amsfonts, indentfirst, empheq}
\usepackage{enumerate,color,bm,graphicx,here}
\usepackage{multicol}
\makeatletter
\def\@cite#1#2{[{{\bfseries #1}\if@tempswa , #2\fi}]}
\renewcommand{\section}{%
\@startsection{section}{1}{\z@}
{0.5truecm plus -1ex minus -.2ex}%
{1.0ex plus .2ex}{\bfseries\large}}
\def\@seccntformat#1{\csname the#1\endcsname.\ }
\makeatother

\numberwithin{equation}{section} 
\theoremstyle{theorem}
\newtheorem{thm}{Theorem}[section]

\newtheorem{lem}[thm]{Lemma}

\theoremstyle{definition}

\newtheorem{remark}{Remark}[section]
\newtheorem*{prth1.1}{Proof of Theorem 1.1}
\newtheorem*{prth1.2}{Proof of Theorem 1.2}

\newcommand{\ep}{\varepsilon}
\newcommand{\pa}{\partial}
\newcommand{\R}{\mathbb{R}}
\newcommand{\N}{\mathbb{N}}
\newcommand{\cl}[1]{{\overline#1}}
\newcommand{\Ombar}{\cl{\Omega}}
\newcommand{\Radon}{\mathcal{M}(\Ombar)}

\newcommand{\ue}{u_{\ep}}
\newcommand{\uie}{u_{0\ep}}
\newcommand{\ve}{v_{\ep}}
\newcommand{\vie}{v_{0\ep}}

\newcommand{\io}{\int_{\Omega}}

\newcommand{\dwto}{\searrow}
\newcommand{\phsp}[2]{C^{#1, \frac{#1}{2}} (\Ombar \times [#2])}
\newcommand{\phhsp}[2]{C^{2 + #1, 1 + \frac{#1}{2}} (\Ombar \times [#2])}
\newcommand{\phtsp}[2]{C^{1 + #1, \frac{1 + #1}{2}} (\Ombar \times [#2])}

\newcommand{\nrm}[2]{\| #1 \|_{ #2 (\Omega)}}
\newcommand{\hnrm}[3]{\| #1 \|_{ \phsp{#2}{#3} }}
\newcommand{\hhnrm}[3]{\| #1 \|_{ \phhsp{#2}{#3} }}
\newcommand{\htnrm}[3]{\| #1 \|_{ \phtsp{#2}{#3} }}
\newcommand{\norm}[2]{\| #1(\cdot, t) \|_{ #2 (\Omega)}}

\newcommand{\nnorm}[4]{\| #1 \|_{ #2 (#3; #4 (\Omega))}}
\newcommand{\embd}{\hookrightarrow}

\newcommand{\intd}[1]{{\rm d}#1}

\newcommand{\wsc}{\stackrel{\star}{\rightharpoonup}}
\newcommand{\til}[1]{\widetilde{#1}}
\newcommand{\altil}{\til{\alpha}}

\newcommand{\utile}{\til{u_{\ep}}}
\newcommand{\vtile}{\til{v_{\ep}}}

\newcommand{\three}{I\hspace{-1.2pt}I\hspace{-1.2pt}I}
\newcommand{\four}{I\hspace{-1.2pt}V}

\DeclareBoldMathCommand{\bmass}{1}
\begin{document}
\footnote[0]
    {2020{\it Mathematics Subject Classification}\/. 
    Primary: 35B65; Secondary: 35Q92, 35A09, 92C17.
    }
%==========================title==========================
\begin{center}
    \Large{{\bf Instantaneous regularization of measure-valued population densities in a Keller--Segel system\\ with flux limitation}}
\end{center}
\vspace{5pt}
%=========================author=========================
\begin{center}
    Shohei Kohatsu
   \footnote[0]{
    E-mail: 
    {\tt sho.kht.jp@gmail.com}
    }\\
    \vspace{12pt}
    Department of Mathematics, 
    Tokyo University of Science\\
    1-3, Kagurazaka, Shinjuku-ku, 
    Tokyo 162-8601, Japan\\
    \vspace{2pt}
\end{center}
\begin{center}    
    \small \today
\end{center}

\vspace{2pt}
%=====================  Abstract  =======================
\newenvironment{summary}
{\vspace{.5\baselineskip}\begin{list}{}{%
     \setlength{\baselineskip}{0.85\baselineskip}
     \setlength{\topsep}{0pt}
     \setlength{\leftmargin}{12mm}
     \setlength{\rightmargin}{12mm}
     \setlength{\listparindent}{0mm}
     \setlength{\itemindent}{\listparindent}
     \setlength{\parsep}{0pt}
     \item\relax}}{\end{list}\vspace{.5\baselineskip}}
\begin{summary}
{\footnotesize {\bf Abstract.}
This paper is concerned with the Keller--Segel system with flux limitation,
 \begin{align}
   \begin{cases}
    u_t=\Delta u - \nabla \cdot (uf(|\nabla v|^{2})\nabla v),
  \\
    v_t=\Delta v - v + u
   \end{cases} \tag{$\ast$}
 \end{align}
in bounded $n$-dimensional domains
with homogeneous Neumann boundary conditions,
where $f$ generalizes the prototype obtained on letting
\[
    f(\xi) = k_f(1 + \xi)^{-\alpha}, \quad \xi \ge 0,
\]
with $k_f > 0$ and $\alpha > 0$.
In this framework, it is shown that
if either $n = 1$ and $\alpha > 0$ is arbitrary,
or $n \ge 2$ and $\alpha > \frac{n-2}{2(n-1)}$,
then for any nonnegative initial data
belonging to the space of Radon measures for the population density
and to $W^{1,q}$ with $q \in (\max\{1, (1-2\alpha)n\}, \frac{n}{n-1})$
for the signal density,
there exists a global classical solution of
the Neumann problem for $(\ast)$,
which is continuous at $t = 0$ in an appropriate sense.
} 
\end{summary}
\vspace{10pt}

\newpage
%%==============================================================%%
%%==============================================================%%
%%                                               Section 1                                              %%
%%                                             Introduction                                            %%
%%==============================================================%%
%%==============================================================%%
\section{Introduction}\label{Sec:Intro}

It is well-known that in many chemotaxis systems,
suitable regularity assumptions on initial data
may lead to global existence of classical solutions,
or blow-up of smooth solutions.
%or existence of smooth solutions that blow up in finite or infinite time.
For example,
%
%Since Keller and Segel~\cite{KS-1970, KS-1971}
%introduced models describing a chemotactic movement of cells,
%a large number of mathematical systems modelling chemotaxis
%have been studied in the last decades.
%In the study of chemotaxis systems, one of the typical settings is that
%the initial data is continuous,
%which will usually lead to constructing classical solutions of systems.
%For example, 
%the Neumann problem for
the chemotaxis system
introduced by Keller and Segel~(\cite{KS-1970, KS-1971}),
\begin{align}\label{KS}
  \begin{cases}
    u_t=\Delta u - \nabla \cdot (u \nabla v),
    & x \in \Omega, \  t > 0,
  \\
    v_t = \Delta v - v + u,
    & x \in \Omega, \ t > 0
  \end{cases}
\end{align}
in a smoothly bounded domain $\Omega \subset \R^n$ $(n \in \N)$,
admits a global bounded classical solution
for any nonnegative initial data
$u(\cdot, 0) = u_0 \in C^0(\Ombar)$ and
$v(\cdot, 0) = v_0 \in
\bigcup_{\sigma > 1} W^{1,\sigma}(\Omega)$
%W^{1, \sigma}(\Omega)$ $(\sigma > 1)$
if $n = 1$
(\cite{HW-2005}),
for $u_0, v_0 \in
\bigcup_{\gamma \in (0,1]} W^{1+ \gamma, 2}(\Omega)$
%W^{1 + \gamma, 2}(\Omega)$
%\subset C^0(\Ombar)$
%$(\gamma \in (0,1])$
satisfying $\io u_0 < 4 \pi$
if $n = 2$
(\cite{NSY-1997}),
and the same result holds
in two- or higher-dimensional settings, provided that
%even if $n \ge 2$, as long as
$u_0 \in C^0(\Ombar)$,
$v_0 \in \bigcup_{\sigma > n}W^{1,\sigma}(\Omega)$
%$v_0 \in W^{1, \sigma}(\Omega)$ $(\sigma > n)$
and $\nrm{u_0}{L^{\frac{n}{2}}} + \nrm{\nabla v_0}{L^n}$ is sufficiently small
(\cite{C-2015}).
On the other hand,
%whereas
when $\Omega$ is a ball in $\R^n$ $(n \ge 3)$,
then there are radial symmetric
positive initial data $u_0 \in C^0(\Ombar)$ and
$v_0 \in W^{1, \infty}(\Omega)$
such that the corresponding solution to \eqref{KS} blows up in
finite time
(\cite{W-2013}).
Similar results regarding global existence and finite-time blow-up in
the parabolic--elliptic counterpart of \eqref{KS} can be found in
\cite{JL-1992, N-2001, S-2005}, for instance.

%On the other hand
Nevertheless, weaker assumptions on initial data still allow us
to find solutions of \eqref{KS}.
In the one-dimensional setting, for any nonnegative initial data
$u_0 \in L^2(\Omega)$ and $v_0 \in W^{1,2}(\Omega)$
the system \eqref{KS} has a unique global bounded solution
which is continuous with values in said spaces
(\cite{OY-2001}).
Concerning local solvability, even more rough initial data can be chosen,
for instance, Biler~\cite{B-1998} considered the case when $u_0$ is a finite
signed measure on $\Omega \subset \R^2$,
and obtained measure-valued solutions which are continuous at $t = 0$.
%with measure values.
The case with measure-valued initial data was also considered for
the parabolic--elliptic system variant
(e.g.\ \cite{SS-2002, LSV-2012}),
and also when $\Omega = \R^2$
(e.g.\ \cite{R-2009, BM-2014, BZ-2015}).
We note that these papers did not consider classical solutions
but more generalized solutions.
%, including mild solutions.

Focusing on global solvability in such systems,
%we find that as in the result of \cite{OY-2001},
%the result of \cite{OY-2001} implies
we know that the
initial datum $u_0$ for the population density does not have to be continuous
(e.g.\ \cite{OY-2001}).
In general, the continuity of initial data is used to construct local solutions
that solve systems in the classical sense,
and in deriving {\it a priori} estimates for the solutions
they only require that the initial data belong to some Lebesgue spaces 
%it seems that it is sufficient to let the initial data belong to some Lebesgue spaces
(e.g.\ \cite{HW-2005, C-2007, GST-2016, L-2023}).
Then the following question arises:
\[
\mbox{\emph{How far the regularity for initial data can be relaxed in finding
global smooth solutions?}}
\]
For the system \eqref{KS} with $n = 1$, this question had been answered by
Winkler~\cite{W-2019},
where they showed existence of the functions
$u, v \in C^{\infty}(\Ombar \times (0, \infty))$,
which solves the Neumann problem for \eqref{KS} in the classical sense,
even when the pair of initial data $(\mu_0, v_0)$ belongs to
$\Radon \times L^2(\Omega)$,
and moreover, the functions $u, v$ are continuous at $t = 0$ in the sense that
\[
    u(\cdot, t) \wsc \mu_0
    \quad\mbox{in}\ \Radon
    \quad\mbox{and}\quad
    v(\cdot, t) \to v_0
    \quad\mbox{in}\ L^2(\Omega)
\]
as $t \dwto 0$.
Here, $\mathcal{M}(\Ombar)$ denotes the space of Radon measures on $\Ombar$,
and throughout the paper we use the identification
$\mathcal{M}(\Ombar) \simeq (C^0(\Ombar))^{*}$ in writing
$\mu(\psi) = \int_{\Ombar} \psi \, \intd{\mu}$ for arbitrary
$\mu \in \Radon$ and $\psi \in C^0(\Ombar)$,
according to the Riesz representation theorem
(see \cite[Theorem 4.31]{B-2011}).
This immediate smoothing was also studied
in the two-dimensional parabolic--elliptic setting by Heihoff~\cite{H-2023},
where they obtained the global smooth solutions with initial data
which belong to $\Radon$ and satisfy a smallness condition.
Similar results for other chemotaxis systems
can be found in \cite{WWX-2021, L-2021, H-2023P}.

We note that existence of classical solutions with measure-valued initial data
has also been considered in semilinear heat equation (\cite{T-2016}, see also
\cite[pp.\ 89--90]{QS-2019}).
%and in the other parabolic equations with even more general initial data
%(e.g.\ \cite{KY-1997}).

Motivated by the previous works,
the purpose of the present paper now consists in studying how the nonlinear sensitivity
in the chemotactic term will affect this immediate regularization.
Specifically, we consider the
%Neumann problem for
Keller--Segel system with flux limitation,
\begin{align}\label{Sys:Main}
  \begin{cases}
    u_t=\Delta u - \nabla \cdot (u f(|\nabla v|^2) \nabla v)
    & \mbox{in}\ \Omega \times (0, \infty),
  \\
    v_t = \Delta v - v + u
    & \mbox{in}\ \Omega \times (0, \infty),
  \\
    \nabla u \cdot \nu=\nabla v \cdot \nu=0
    & \mbox{on}\ \pa\Omega \times (0, \infty),
  \\
    u(\cdot ,0) = \mu_0, \ v(\cdot, 0) = v_0
    & \mbox{in}\ \Omega
  \end{cases}
\end{align}
in a smoothly bounded domain $\Omega \subset \R^n$ $(n \in \N)$,
where
\begin{equation}\label{Ass:f}
    f \in C^2([0, \infty)),
\end{equation}
where $\nu$ is the outward normal vector to $\pa\Omega$,
and where $\mu_0, v_0$ are given nonnegative functions.

It was established in \cite[Proposition 1.2]{W-2022}
that global classical solutions of the system \eqref{Sys:Main}
%was established in \cite[Proposition 1.2]{W-2022}
with initial data $\mu_0 \in C^0(\Ombar)$ and
$v_0 \in \bigcup_{r > n} W^{1,r}(\Omega)$ exist,
under the condition that the function $f$ satisfies
\begin{align}\label{Ass:f:Main}
    f(\xi) \le k_f (1 + \xi)^{-\alpha}\quad\mbox{for all}\ \xi \ge 0
\end{align}
with some $k_f > 0$ and $\alpha \in \R$ fulfilling
\[
  \begin{cases}
        \alpha \in \R
        & \mbox{if}\ n = 1,
\\ 
        \alpha > \dfrac{n-2}{2(n-1)}
        &\mbox{if}\ n \ge 2.
  \end{cases}        
\]
On the other hand, when $\Omega$ is a ball in $\R^n$ $(n\ge 3)$,
it was shown in \cite{W-2022b} that
there exists a suitable initial data $\mu_0 \in C^0(\Ombar)$ such that
the corresponding smooth solution blows up in finite time,
provided that $f$ fulfills
\[
    f(\xi) \ge k_f (1 + \xi)^{-\alpha}\quad\mbox{for all}\ \xi \ge 0
\]
with some $k_f > 0$ and $\alpha \in (0, \frac{n-2}{2(n-1)})$.
%Recently, the case with logistic source was considered in \cite{Z-2023, MVY-2023},
Similar results regarding global existence and finite-time blow-up
in variants of \eqref{Sys:Main} were obtained in \cite{NT-2018, T-2022},
and also the system \eqref{Sys:Main} with logistic source
was considered in \cite{Z-2023, MVY-2023}.
However,
in all these results initial data were assumed to be continuous,
and it can be expected that the regularity requirements for initial data
should be relaxed, as in the system \eqref{KS}.

%It is expected that 
Apart from an improvement over \cite{W-2022} regarding the regularity of
solutions, we shall weaken the requirements on initial data and obtain a global
classical solution of \eqref{Sys:Main}.
%Our aim of this paper is to construct global smooth solutions of
%\eqref{Sys:Main} with

Our main result reads as follows.

%
%----------------------------------------------------------------%
%----------------------------------------------------------------%
%                                               Theorem 1.1                                             %
%----------------------------------------------------------------%
%----------------------------------------------------------------%
%
\begin{thm}\label{Thm:Main}
Let $n \in \N$, and
let $f$ satisfy \eqref{Ass:f} as well as \eqref{Ass:f:Main}
with some $k_f > 0$ and $\alpha > 0$ fulfilling
\[
  \begin{cases}
        \alpha > 0
        & \mbox{if}\ n = 1,
\\ 
        \alpha > \dfrac{n-2}{2(n-1)}
        &\mbox{if}\ n \ge 2.
  \end{cases}        
\]
Suppose that $\mu_0$ and $v_0$ satisfy
\begin{align}\label{Ass:ini}
    \begin{cases}
    \mu_0 \in \Radon, \quad \mu_0 \ge 0 \ \mbox{in}\ \Ombar
    \quad\mbox{and}\quad \mu_0 \neq 0,
  \\
    v_0 \in W^{1,q}(\Omega), \quad v_0 \ge 0 \ \mbox{in}\ \Omega
  \end{cases}
\end{align}
with
\begin{align}\label{Ass:q}
%    q \in \left(\max\{1, (1-2\alpha)n\}, \frac{n}{n-1}\right).
  \begin{cases}
        q \in (\max\{1, (1-2\alpha)n\}, \infty)
        & \mbox{if}\ n = 1,
\\ 
        q \in \left(\max\{1, (1-2\alpha)n\}, \dfrac{n}{n-1}\right)
        &\mbox{if}\ n \ge 2.
  \end{cases}      
\end{align}
Then there exists a pair of nonnegative functions
\begin{align}\label{Result:reg}
    (u,v) \in \big( C^{\infty}(\Ombar \times (0, \infty)) \big)^2
\end{align}
such that $(u, v)$ solves
the boundary value problem in the system \eqref{Sys:Main}
in the classical sense in
$\Ombar \times (0,\infty)$,
%that $u \ge 0$ and $v \ge 0$ in $\Ombar \times (0, \infty)$,
that satisfies
\begin{equation}\label{Result:mass}
    \io u(\cdot, t) = m := \mu_0(\bmass)
    \quad\mbox{for all}\ t > 0
\end{equation}
with $\bmass(x) := 1$ for $x \in \Ombar$,
and that
\begin{alignat}{2}
    &u(\cdot, t) \wsc \mu_0
    \quad\mbox{in}\ \Radon\quad& &\mbox{as}\ t \dwto 0\label{Result:u}
    \\
\intertext{as well as}
    &v(\cdot, t) \to v_0
    \quad\mbox{in}\ W^{1,q}(\Omega)\quad& &\mbox{as}\ t\dwto 0.\label{Result:v}
\end{alignat}
\end{thm}
%----------------------------------------------------------------%

Our approach for constructing global smooth solutions of \eqref{Sys:Main},
as in previous studies,
is to firstly approximate by smooth solutions $(\ue, \ve)$
which solve the regularized problem \eqref{Sys:Reg}
%of \eqref{Sys:Main}
with smooth initial data $\uie, \vie$
(Section~\ref{Sec:Inside}),
and to secondly establish continuity of
the constructed solutions at $t = 0$ regarding the topology claimed in
%\eqref{Result:u} as well as \eqref{Result:v}.
Theorem~\ref{Thm:Main}
(Section~\ref{Sec:t0}).

In the first step, we will derive an $L^{\infty}$-estimate for $\ue(\cdot, t)$
on $(\tau, T)$ with $0 < \tau < T$,
which is independent of $\ep$
(Lemma~\ref{Lem:ueLinfty}),
whence the standard parabolic regularity theory from \cite{LSU-1968, L-1987}
asserts suitable H\"{o}lder bounds for $\ue$ and $\ve$
(Lemmas~\ref{Lem:veHol2} and \ref{Lem:ueHol2}),
which allow us to construct such solution candidates of \eqref{Sys:Main}
along a suitable sequence
via compact embedding property of H\"{o}lder spaces
(Lemma~\ref{Lem:uvC21}).

In order to prove continuity \eqref{Result:u}, it is sufficient to show that
for any $\varphi \in C^0(\Ombar)$
the function $t \mapsto |\io u(\cdot, t) \varphi - \mu_0 (\varphi)|$ tends to $0$
%
%\[
%    \left| \io u(\cdot, t) \varphi - \mu_0 (\varphi)\right| \to 0
%    \quad \mbox{as}\ t \dwto 0
%\]
%
as $t \dwto 0$
(Lemma~\ref{Lem:Limit_u}),
and this will be achieved on the basis of uniform continuity of $\ue$
in the sense of the weak-*-topology of $\Radon$
(Lemma~\ref{Lem:ueuie}).
To this end, we shall establish the uniform time integrability of
$\io \ue(\cdot, t)|\nabla \ve(\cdot, t)|^{1-2\alpha}$ near $t = 0$
(Lemma~\ref{Lem:taxisL1}),
which then gives us a similar integrability property for
$\io (\ue)_t(\cdot, t) \varphi$ in Lemma~\ref{Lem:ueuie}.
The proof of \eqref{Result:v} in Lemma~\ref{Lem:Limit_v}
relies on semigroup based approaches as in \cite{H-2023P}.
%$t \mapsto$

The key to the proofs of all two of the above steps
lies in establishing the uniform time integrability of $\ue(\cdot, t)$
in $L^r(\Omega)$ with some $r > 1$ near $t = 0$
(Lemma~\ref{Lem:ueLr-0}).
Due to the quite low regularity of $u_0$,
we cannot obtain uniform bounds for $\nrm{\uie}{L^r}$ for any $r > 1$,
but still we have an $\ep$-independent bound for $\nrm{\uie}{L^1}$
(see \eqref{uie:mass}),
and thus we use semigroup methods instead of testing procedures.
The most important estimate in the proof will be an $L^s$-estimate for
the function
$t\mapsto \ue(\cdot, t)|\nabla \ve(\cdot, t)|^{1-2\alpha}$,
arising from the flux limitation of the form in \eqref{Ass:f:Main},
on $(0, T)$ with $T > 0$ and a suitable choice of $s > 1$
as in Lemma~\ref{Lem:param-0},
where the largeness assumption on $\alpha$ from Theorem~\ref{Thm:Main} is required.

%%==============================================================%%
%%==============================================================%%
%%                                               Section 2                                              %%
%%                                             Preliminaries                                           %%
%%==============================================================%%
%%==============================================================%%
\section{Preliminaries}\label{Sec:Pre}

In this section we collect some $L^p$-$L^q$ estimates
for the Neumann heat semigroup and inhomogeneous linear heat equations.
%Let us first recall an $L^p$-$L^q$ estimate for solutions
%of inhomogeneous linear heat equations.
Throughout the sequel, by $(e^{t\Delta})_{t \ge 0}$ we will denote
the Neumann heat semigroup in the domain $\Omega$.
The following statement in this regard extracts from
\cite[Lemma 2.1]{IY-2020} what will be needed here.
%
%%%%%%%%%%%%%%%%%%%%%%%%%%%%%%%%%%%%%%%%%%%%%%%%%%%%%%%%%%%%%%%%%%
%                                                 Lemma 2.1                                             %
%%%%%%%%%%%%%%%%%%%%%%%%%%%%%%%%%%%%%%%%%%%%%%%%%%%%%%%%%%%%%%%%%%
%
\begin{lem}\label{Lem:Inhom}
Let $T > 0$, and let $\Omega \subset \R^n$ $(n \in \N)$ be a bounded domain.
Consider the linear heat equation
\begin{align}\label{Sys:Inh}
  \begin{cases}
    z_t=\Delta z - z + w
    & \mbox{in}\ \Omega \times (0, T),
  \\
    \nabla z \cdot \nu = 0
    & \mbox{on}\ \pa\Omega \times (0, T),
  \\
    z(\cdot ,0) = z_0
    & \mbox{in}\ \Omega.
  \end{cases}
\end{align}
Let $1 \le q \le p \le \infty$ such that
$\frac{1}{q} - \frac{1}{p} < \frac{1}{n}$,
and suppose that $z_0 \in W^{1,p}(\Omega)$ as well as
$w \in L^{\infty}(0, T; L^q(\Omega))$.
Then there exists a unique solution $z \in C^0([0,T]; W^{1,p}(\Omega))$
to \eqref{Sys:Inh} given by
\[
    z(t) = e^{-t} e^{t \Delta} z_0
    + \int_0^t e^{-(t-\sigma)} e^{(t-\sigma) \Delta} w(\sigma)
      \, \intd{\sigma}
    \quad\mbox{for}\ t \in [0, T],
\]
%
%where $(e^{t \Delta})_{t \ge 0}$ denotes the Neumann heat semigroup in $\Omega$,
and moreover,
\[
    \nrm{\nabla z(t)}{L^p}
    \le e^{-t} \nrm{\nabla z_0}{L^p}
    + \Gamma\left(\frac{1}{2}
    - \frac{n}{2}\left(\frac{1}{q} - \frac{1}{p}\right)\right)
    \nnorm{w}{L^{\infty}}{0,t}{L^q}
%    \quad\mbox{for all}\ t \in [0, T],
\]
for all $t \in [0, T]$, where $\Gamma$ is the Gamma function.
\end{lem}
%-----------------------------------------------------------------%

Let us also recall some estimates as given in
\cite[Lemma 1.3]{W-2010} and \cite[Lemma 2.1]{C-2015},
which will be used in several places.
We henceforth letting $\lambda_1 > 0$ denote
the first nonzero eigenvalue of $- \Delta$ in $\Omega$
under the Neumann boundary condition.
%
%%%%%%%%%%%%%%%%%%%%%%%%%%%%%%%%%%%%%%%%%%%%%%%%%%%%%%%%%%%%%%%%%%
%                                                 Lemma 2.2                                             %
%%%%%%%%%%%%%%%%%%%%%%%%%%%%%%%%%%%%%%%%%%%%%%%%%%%%%%%%%%%%%%%%%%
%
\begin{lem}\label{Lem:Semi}
Let $\Omega \subset \R^n$ $(n\in\N)$ be a smoothly bounded domain.
%and
%let $\lambda_1 > 0$ be the first nonzero eigenvalue of $- \Delta$ in $\Omega$
%under the Neumann boundary condition.
Then there exists a constant $C > 0$ which depends only on $\Omega$
and which has the following properties\/{\rm :}
\begin{enumerate}[{\rm (i)}]
%%%%%%%%%%%%%%%%%%%%%%%%%%%%%%%%
\item\label{Semi-1} If $1 \le q \le p \le \infty$, then
\[
    \nrm{e^{t \Delta} w}{L^p}
    \le C(1 + t^{-\frac{n}{2}(\frac{1}{q} - \frac{1}{p})}) e^{-\lambda_1 t}
    \nrm{w}{L^q}
    \quad\mbox{for all}\ t > 0
\]
holds for all $w \in L^q(\Omega)$ with $\io w = 0$.
%%%%%%%%%%%%%%%%%%%%%%%%%%%%%%%%
\item\label{Semi-2} If $1\le q \le p\le \infty$, then
\[
    \nrm{\nabla e^{t\Delta}w}{L^p}
      \le C(1 + t^{-\frac{1}{2}-\frac{n}{2}(\frac{1}{q}-\frac{1}{p})})
      e^{-\lambda_1 t} \nrm{w}{L^q}
    \quad\mbox{for all}\ t > 0
\]
holds for any $w \in L^q(\Omega)$.
%%%%%%%%%%%%%%%%%%%%%%%%%%%%%%%%
\item\label{Semi-3} If $1 < q \le p < \infty$ or $1 < q < p = \infty$, then
\[
    \nrm{e^{t\Delta}\nabla\cdot w}{L^p}
      \le C(1 + t^{-\frac{1}{2}-\frac{n}{2}(\frac{1}{q}-\frac{1}{p})})
      e^{-\lambda_1 t} \nrm{w}{L^q}
    \quad\mbox{for all}\ t > 0
\]
is valid for all $w \in (L^q(\Omega))^n$,
where $e^{t\Delta}\nabla \cdot$ is the extension of
the corresponding operator on $(C_{\rm c}^{\infty}(\Omega))^n$ to
a continuous operator from $(L^q(\Omega))^n$ to $L^p(\Omega)$.
\end{enumerate}
\end{lem}
%-----------------------------------------------------------------%
%
%\blue{As a last preliminary, let us furthermore recall from
%\cite[Section 2]{HW-2005}, which will be used in Lemma~\ref{Lem:Limit_v}
%to prove continuity of $v$ at $t = 0$.}
%
%%%%%%%%%%%%%%%%%%%%%%%%%%%%%%%%%%%%%%%%%%%%%%%%%%%%%%%%%%%%%%%%%%
%                                                 Lemma 2.3                                             %
%%%%%%%%%%%%%%%%%%%%%%%%%%%%%%%%%%%%%%%%%%%%%%%%%%%%%%%%%%%%%%%%%%
%
%\begin{lem}\label{Lem:FracS}
%Assume that $\Omega \subset \R^n$ $(n \in \N)$ is a smoothly bounded domain.
%Then there are $C > 0$ and $\mu > 0$ with the following property\/{\rm :}
%If $1 \le q \le p < \infty$, then for any $w \in L^q(\Omega)$,
%
%\[
%    \nrm{(-\Delta + 1)^{s} e^{t \Delta} w}{L^p}
%      \le Ct^{-s - \frac{n}{2}(\frac{1}{q}-\frac{1}{p})} e^{(1-\mu)t}
%      \nrm{w}{L^q}
%\]
%
%holds for all $t > 0$ and $s \ge 0$.
%\end{lem}
%-----------------------------------------------------------------%

%%==============================================================%%
%%==============================================================%%
%%                                               Section 3                                              %%
%%                   Estimates for approximate solutions near t=0                         %%
%%==============================================================%%
%%==============================================================%%
\section{Estimates for approximate solutions near $t = 0$}\label{Sec:Esti0}

In order to regularize the system \eqref{Sys:Main},
we begin by approximating initial data $\mu_0$ and $v_0$ given in \eqref{Ass:ini}
with \eqref{Ass:q}.
Let us first fix a family of functions
$(\uie)_{\ep \in (0,1)} \subset C^{\infty}(\Ombar)$ such that
$\uie \ge 0$ in $\Ombar$, that
\begin{align}
    &\nrm{\uie}{L^1} = \mu_0(\bmass) = m
      \quad\mbox{for all}\ \ep \in (0,1),\label{uie:mass}
      \\
      \intertext{and that}
    &\uie \wsc \mu_0\quad\mbox{in}\ \Radon\quad\mbox{as}\ \ep\dwto 0,
    \label{uie:Radon}
\end{align}
noting that such a family of functions can be found in
\cite[pp.\ 470--471]{B-2011}.
Following \cite{H-2023P},
we next introduce a family of functions
$(\vie)_{\ep \in (0,1)} \subset C^{\infty}(\Ombar)$
defined by
%in a similar way as in \cite{H-2023P}:
%
\begin{align}\label{Def:vie}
    \vie := e^{\ep (\Delta - 1)} v_0
      \quad\mbox{for}\ \ep\in (0,1).
\end{align}
According to the maximum principle,
$\vie \ge 0$ in $\Ombar$ for all $\ep\in (0,1)$.
Moreover, due to continuity of the semigroup at $t=0$,
we have
%the family $(\vie)_{\ep\in (0,1)}$ has the convergence property,
%
\begin{align}\label{Conv_vie}
    \vie \to v_0 \quad\mbox{in}\ W^{1,q}(\Omega)
      \quad\mbox{as}\ \ep \dwto 0.
\end{align}
We then consider the regularized problem
\begin{align}\label{Sys:Reg}
  \begin{cases}
    (\ue)_t=\Delta \ue - \nabla \cdot (\ue f(|\nabla \ve|^2) \nabla \ve)
    & \mbox{in}\ \Omega \times (0, \infty),
  \\
    (\ve)_t = \Delta \ve - \ve + \dfrac{\ue}{1 + \ep \ue}
    & \mbox{in}\ \Omega \times (0, \infty),
  \\
    \nabla \ue \cdot \nu=\nabla \ve \cdot \nu=0
    & \mbox{on}\ \pa\Omega \times (0, \infty),
  \\
    \ue(\cdot ,0) = \uie, \ \ve(\cdot, 0) = \vie
    & \mbox{in}\ \Omega,
  \end{cases}
\end{align}
where $\ep \in (0, 1)$.
The approximation used here immediately implies existence of
global approximate solutions.
%
%%%%%%%%%%%%%%%%%%%%%%%%%%%%%%%%%%%%%%%%%%%%%%%%%%%%%%%%%%%%%%%%%%
%                                                 Lemma 3.1                                             %
%%%%%%%%%%%%%%%%%%%%%%%%%%%%%%%%%%%%%%%%%%%%%%%%%%%%%%%%%%%%%%%%%%
%
\begin{lem}\label{Lem:App}
For any $\ep \in (0,1)$, there exist uniquely determined functions
\begin{align}\label{Reg_App}
    \begin{cases}
      \ue \in C^0(\Ombar \times [0, \infty))\cap C^{2,1}(\Ombar \times (0, \infty)),\\
      \ve \in C^0(\Ombar \times [0, \infty))
        \cap C^{2,1}(\Ombar \times (0, \infty)) \cap
        \bigcap_{r > n} L^{\infty}_{\rm loc}([0, \infty); W^{1,r}(\Omega))
    \end{cases}
\end{align}
such that $(\ue, \ve)$ is a global classical solution to \eqref{Sys:Reg}.
Moreover, $\ue \ge 0$ and $\ve \ge 0$ in $\Ombar \times [0, \infty)$,
and
\begin{align}\label{Mass_App}
    \io \ue(\cdot, t) = m \quad\mbox{for all}\ t > 0
\end{align}
with $m > 0$ as defined in \eqref{Result:mass}.
\end{lem}
%%%%%%%%%%%%%%%%%%%%%%%%%%%%%%%%%%%%%%%%%%%%%%%%%%%%%%%%%%%%%%%%%%
%------------------------------proof------------------------------%
\begin{proof}
A standard fixed point argument and regularity theory
as in \cite[Lemma 1.1]{W-2010L}),
together with the fact that
$\frac{\ue}{1 + \ep\ue} \le \frac{1}{\ep}$ for all $\ep \in (0,1)$,
provide existence and uniqueness of a global classical solution
$(\ue, \ve)$ to \eqref{Sys:Reg} that satisfies \eqref{Reg_App}.
Nonnegativity of $\ue$ and $\ve$ result from the maximum principle,
while the mass conservation property \eqref{Mass_App} follows by
integrating the first equation of \eqref{Sys:Reg}
over $\Omega$.
\end{proof}
%-----------------------------------------------------------------%

We henceforth fix the unique classical solution $(\ue, \ve)$
to \eqref{Sys:Reg} for each $\ep \in (0,1)$.
A necessary first observation is uniform boundedness of
$\nabla \ve$ in $L^{\infty}(0, \infty; L^p(\Omega))$ with $p \in (0, q]$.
%
%%%%%%%%%%%%%%%%%%%%%%%%%%%%%%%%%%%%%%%%%%%%%%%%%%%%%%%%%%%%%%%%%%
%                                                 Lemma 3.2                                             %
%%%%%%%%%%%%%%%%%%%%%%%%%%%%%%%%%%%%%%%%%%%%%%%%%%%%%%%%%%%%%%%%%%
%
\begin{lem}\label{Lem:vexLp-0}
Let $p \in (0, q]$.
Then there exists a constant $C = C(|\Omega|, p, q, m, n) > 0$ such that
\begin{align}\label{vexLp-0}
    \norm{\nabla \ve}{L^p} \le C
      \quad\mbox{for all}\ t > 0\ \mbox{and}\ \ep \in (0,1).
\end{align}
\end{lem}
%%%%%%%%%%%%%%%%%%%%%%%%%%%%%%%%%%%%%%%%%%%%%%%%%%%%%%%%%%%%%%%%%%
%------------------------------proof------------------------------%
\begin{proof}
Without loss of generality assuming that $p \in [1,q]$,
we note that
\[%\label{vexLp-0:finite}
    1 - \frac{1}{p} < \frac{1}{n}
\]
holds true.
Indeed, if $n = 1$ this is obvious,
while if $n \ge 2$ we infer from \eqref{Ass:q} that $\frac{1}{q} > \frac{n-1}{n}$,
and that thus
\[
    1 - \frac{1}{p} \le 1 - \frac{1}{q} < 1 - \frac{n-1}{n} = \frac{1}{n}.
\]
Now, thanks to \eqref{Conv_vie} there is $c_1 > 0$ such that
\begin{align}\label{vexLp-0:vie}
    \nrm{\nabla \vie}{L^q} \le c_1 \quad\mbox{for all}\ \ep \in (0,1).
\end{align}
Then we may apply Lemma~\ref{Lem:Inhom}
along with the H\"{o}lder inequality, \eqref{vexLp-0:vie}
and \eqref{Mass_App} to estimate
\begin{align*}
    \norm{\nabla \ve}{L^p}
    &\le e^{-t} \nrm{\nabla \vie}{L^p}
    + \Gamma\left(\frac{1}{2}-\frac{n}{2}\left(1 - \frac{1}{p}\right)\right)
    \left\|\frac{\ue}{1 + \ep\ue}\right\|_{
      L^{\infty}(0, t; L^1(\Omega))}
\\
    &\le c_1 |\Omega|^{\frac{1}{p}-\frac{1}{q}}
    + \Gamma\left(\frac{1}{2}-\frac{n}{2}\left(1 - \frac{1}{p}\right)\right)
    \nnorm{\ue}{L^{\infty}}{0,t}{L^1}
\\
    &= c_1 |\Omega|^{\frac{1}{p}-\frac{1}{q}}
    + \Gamma\left(\frac{1}{2}-\frac{n}{2}\left(1 - \frac{1}{p}\right)\right)m
    \quad\mbox{for all}\ t > 0\ \mbox{and}\ \ep \in (0,1),
%\\
%    &\le |\Omega|^{\frac{1}{p}-\frac{1}{q}} \sup_{\ep\in (0,1)}\nrm{\nabla \vie}{L^q}
%    + \Gamma\left(\frac{1}{2}-\frac{n}{2}\left(1 - \frac{1}{p}\right)\right)m
\end{align*}
from which \eqref{vexLp-0} can readily be derived.
%
%for all $t > 0$ and $\ep \in (0,1)$.
%Since $c_2$ is finite due to \eqref{vexLp-0:finite},
\end{proof}
%-----------------------------------------------------------------%

A next step of key importance is deriving an $\ep$-independent
$L^r$-estimate for $\ue$ near $t = 0$
with some $r > 1$,
which will be used in several places below
firstly in order to establish estimates away from $t = 0$ that help to construct
solution candidates,
and secondly to show that said candidates are continuous at $t = 0$
in an appropriate sense.

To launch our derivation, we state a preparatory elementary argument
concerning a suitable parameter selections that will determine essential parts
of our approach.
%
%%%%%%%%%%%%%%%%%%%%%%%%%%%%%%%%%%%%%%%%%%%%%%%%%%%%%%%%%%%%%%%%%%
%                                                 Lemma 3.3                                             %
%%%%%%%%%%%%%%%%%%%%%%%%%%%%%%%%%%%%%%%%%%%%%%%%%%%%%%%%%%%%%%%%%%
%
\begin{lem}\label{Lem:param-0}
Let $\alpha > 0$ satisfy
\begin{align}\label{param-0:alpha}
    \begin{cases}
      \alpha \in (0, \frac{1}{2}) & \mbox{if}\ n = 1, \\
      \alpha \in (\frac{n-2}{2(n-1)}, \frac{1}{2}) & \mbox{if}\ n \ge 2,
    \end{cases}
\end{align}
and assume that $r$ satisfies
\begin{align}\label{param-0:r}
    \begin{cases}
      r \in (\frac{q}{q-1+2\alpha}, \infty] & \mbox{if}\ n = 1, \\
      r \in (\frac{q}{q-1+2\alpha}, \infty) & \mbox{if}\ n = 2, \\
      r \in (\frac{q}{q-1+2\alpha}, \frac{n}{n-2}) & \mbox{if}\ n \ge 3.
    \end{cases}
\end{align}
Then there exist $s, s_1, s_2 > 1$ with
%
%\begin{align*}
%    \begin{cases}
%    &s_1 \in (s, r)\quad\mbox{and}\\
%    &s_2(1- 2\alpha) \in (0, q]
%    \end{cases}
%\end{align*}
%
%such that
%
%\begin{align*}
%    &\frac{1}{s}-\frac{1}{r} < \frac{1}{n},
%    \\ %\intertext{and}
%    &\frac{1}{s} = \frac{1}{s_1} + \frac{1}{s_2} \quad\mbox{and}
%    \\ %\intertext{as well as}
%    &1 - \frac{1}{s_1} < \frac{2}{n}.
%\end{align*}
%
the following property\/{\rm :}
\begin{multicols}{2}
\raggedcolumns
    \begin{enumerate}[{\rm (i)}]
    \setlength{\leftskip}{5mm}
    \item $\frac{1}{s}-\frac{1}{r} < \frac{1}{n}$,\label{param-0:1}
    \setcounter{enumi}{2}
    \item $s_2(1 - 2\alpha) \in (0, q]$,\label{param-0:3}
    \setcounter{enumi}{4}
    \item $1 - \frac{1}{s_1} < \frac{2}{n}$.\label{param-0:5}
    \columnbreak
    \setcounter{enumi}{1}
    \item $s_1 \in (s, r)$,\label{param-0:2}
    \setcounter{enumi}{3}
    \item $\frac{1}{s} = \frac{1}{s_1} + \frac{1}{s_2}$,\label{param-0:4}
    \end{enumerate}
\end{multicols}
\end{lem}
%%%%%%%%%%%%%%%%%%%%%%%%%%%%%%%%%%%%%%%%%%%%%%%%%%%%%%%%%%%%%%%%%%
%------------------------------proof------------------------------%
\begin{proof}
Since the relation $q > (1 - 2\alpha) n$ holds because of \eqref{Ass:q},
we can find $\delta_1 > 0$ such that
\begin{align}\label{param-0:delta1}
    \frac{1-2\alpha}{q} = \frac{1}{n} - \delta_1.
\end{align}
Moreover, \eqref{param-0:r} entails
%$\frac{1}{r} < 1 - \frac{1-2\alpha}{q}$,
%
\begin{align}\label{param-0:rela}
    \frac{1}{r} < 1 - \frac{1-2\alpha}{q},
\end{align}
which shows that there is $\delta_2 > \frac{1-2\alpha}{q}$ such that
\begin{align}\label{param-0:delta2}
    \frac{1}{r} = 1 - \delta_2.
\end{align}
Then we see from \eqref{param-0:delta1} that
\[
    \delta_1 + \delta_2 = \frac{1}{n} - \frac{1-2\alpha}{q} + \delta_2
    > \frac{1}{n},
\]
which enables us to pick $\delta_3 > 0$ such that
$\delta_3 < \min\{\delta_1, \delta_2\}$ and
\begin{align}\label{param-0:delta3}
%    \delta_3 < \min\{\delta_1, \delta_2\}
%      \quad\mbox{and}\quad
    \delta_1 + \delta_2 - \delta_3 > \frac{1}{n}.
\end{align}
Now, we put
\begin{align}\label{param-0:s1}
s_1 := \frac{1}{\frac{1}{r} + \delta_3}
\end{align}
and obtain from \eqref{param-0:delta2} that
%
%\[
%    s_1 := \frac{1}{\frac{1}{r} + \delta_3}
%      \quad\mbox{and}\quad
%    s_2 := \frac{q}{1-2\alpha}
%\]
%
%
\[
s_1 = \frac{1}{1 - \delta_2 + \delta_3} > 1.
\]
Furthermore, letting
\begin{align}\label{param-0:s2}
s_2 := \frac{q}{1-2\alpha}
\end{align}
yields $s_2 > n$ from the relation $q > (1-2\alpha)n$.
If we define
\begin{align}\label{param-0:s}
    s := \frac{s_1 s_2}{s_1 + s_2},
\end{align}
we accordingly obtain \eqref{param-0:4}, whence
\eqref{param-0:delta1} and \eqref{param-0:delta2}
together with \eqref{param-0:delta3} entail that
\begin{align*}
%\[
    \frac{1}{s}
    &= \frac{1}{s_1} + \frac{1}{s_2}
\\
    &= 1 + \frac{1}{n} - (\delta_1 + \delta_2 - \delta_3)
\\
    &< 1
%\]
\end{align*}
and therefore $s > 1$.
We claim that the parameters $s, s_1, s_2$ as in
\eqref{param-0:s1}, \eqref{param-0:s2} and \eqref{param-0:s} satisfy
\eqref{param-0:1}, \eqref{param-0:2}, \eqref{param-0:3} and \eqref{param-0:5}.
To see this,
let us first note that \eqref{param-0:3} immediately follows from \eqref{param-0:s2},
that \eqref{param-0:2} holds since $s < s_1$ from \eqref{param-0:s},
and that \eqref{param-0:5} is obvious when $n = 1$. 
Next, invoking \eqref{param-0:4} and \eqref{param-0:delta1} we would see that
\begin{align*}
    \frac{1}{s}-\frac{1}{r}
    &= \frac{1}{s_1} + \frac{1}{s_2} - \frac{1}{r}
\\
    &= \frac{1}{r} + \delta_3 + \frac{1-2\alpha}{q} - \frac{1}{r}
\\
    &= \delta_3 + \frac{1}{n} - \delta_1
\\
    &< \frac{1}{n}
\end{align*}
and hence \eqref{param-0:1}.
Finally, for $n \ge 2$, by \eqref{param-0:r}
we have $\frac{1}{r} > \frac{n-2}{n}$,
and thus
\begin{align*}
    1 - \frac{1}{s_1}
    &= 1 - \frac{1}{r} - \delta_3
\\
    &< 1 - \frac{1}{r}
\\
    &< 1 - \frac{n-2}{n}
\\
    &= \frac{2}{n},
\end{align*}
which proves \eqref{param-0:5}.
\end{proof}
%-----------------------------------------------------------------%
%
%%%%%%%%%%%%%%%%%%%%%%%%%%%%%%%%%%%%%%%%%%%%%%%%%%%%%%%%%%%%%%%%%%
%                                                 Remark 3.1                                             %
%%%%%%%%%%%%%%%%%%%%%%%%%%%%%%%%%%%%%%%%%%%%%%%%%%%%%%%%%%%%%%%%%%
%
\begin{remark}\label{Rem:param-0:r}
%\begin{enumerate}[{\rm (i)}]
%\item
It can be verified that under the assumption in \eqref{Ass:q},
we can actually pick $r$ satisfying \eqref{param-0:r}.
Indeed, for $n \ge 3$, it is sufficient to show that
\[
    \frac{q}{q-1+2\alpha} < \frac{n}{n-2},
\]
which is equivalent to the inequality $q > \frac{n}{2}(1 - 2\alpha)$,
and this is always satisfied when $q$ fulfills \eqref{Ass:q}.
%\end{enumerate}
\end{remark}

%
%%%%%%%%%%%%%%%%%%%%%%%%%%%%%%%%%%%%%%%%%%%%%%%%%%%%%%%%%%%%%%%%%%
%                                                 Remark 3.2                                             %
%%%%%%%%%%%%%%%%%%%%%%%%%%%%%%%%%%%%%%%%%%%%%%%%%%%%%%%%%%%%%%%%%%
%
\begin{remark}\label{Rem:param-0:alpha}
Lemma~\ref{Lem:param-0} requires the upper bound for $\alpha$
as in \eqref{param-0:alpha}, however,
we can still proceed without this restriction on $\alpha$.
More precisely, when $\alpha \ge \frac{1}{2}$,
choosing an arbitrary $\altil$ that satisfies \eqref{param-0:alpha}
so that by \eqref{Ass:f:Main} we have
\begin{align*}
    f(\xi)
    &\le k_f (1 + \xi)^{-\alpha}
\\
    &= k_f (1 + \xi)^{-\altil} (1 + \xi)^{\altil - \alpha}
\\
    &\le k_f (1 + \xi)^{-\altil}
    \quad\mbox{for all}\ \xi \ge 0,
\end{align*}
and thus we can assume that $\alpha$ satisfies \eqref{param-0:alpha}
without loss of generality.
\end{remark}

Having made sure that there are suitable parameters,
we are now in a position to obtain the following estimate for $\ue$.
%
%%%%%%%%%%%%%%%%%%%%%%%%%%%%%%%%%%%%%%%%%%%%%%%%%%%%%%%%%%%%%%%%%%
%                                                 Lemma 3.4                                             %
%%%%%%%%%%%%%%%%%%%%%%%%%%%%%%%%%%%%%%%%%%%%%%%%%%%%%%%%%%%%%%%%%%
%
\begin{lem}\label{Lem:ueLr-0}
Let $r$ satisfy
\begin{align}\label{ueLr-0:r}
    \begin{cases}
      r \in (1, \infty] & \mbox{if}\ n = 1, \\
      r \in (1, \infty) & \mbox{if}\ n = 2, \\
      r \in (1, \frac{n}{n-2}) & \mbox{if}\ n \ge 3,
    \end{cases}
\end{align}
and suppose that $\alpha$ satisfies \eqref{param-0:alpha}.
Then there exists
$\gamma \in (1, \infty]$ satisfying $1 - \frac{1}{\gamma} < \frac{2}{n}$
with the following property\/{\rm :}
For all $T > 0$ there is
$C = C(\alpha, T, q, \gamma, m, n, |\Omega|) > 0$ such that
\begin{gather}
    \norm{\ue}{L^r} \le C t^{-\frac{n}{2}(1 - \frac{1}{\gamma})}
      \quad\mbox{for all}\ t \in (0, T)\ \mbox{and}\ \ep\in (0,1),
    \label{ueLr-0:Ct}
\\ \intertext{and}
%\notag
    \int_{0}^{t} \nrm{\ue(\cdot, \sigma)}{L^r} \,\intd{\sigma}
    \le C t^{1 - \frac{n}{2}(1 - \frac{1}{\gamma})}
      \quad\mbox{for all}\ t \in (0, T)\ \mbox{and}\ \ep\in (0,1).
      \label{ueLr-0:IntCt}
\end{gather}
%
%In particular, if $r$ satisfies \eqref{param-0:r},
%\blue{then 
\end{lem}
%%%%%%%%%%%%%%%%%%%%%%%%%%%%%%%%%%%%%%%%%%%%%%%%%%%%%%%%%%%%%%%%%%
%------------------------------proof------------------------------%
\begin{proof}
We first note that by the H\"{o}lder inequality,
it is sufficient to consider the case when $r$ satisfies \eqref{param-0:r}.
Then Lemma~\ref{Lem:param-0} provides $s, s_1, s_2 > 1$
fulfilling \eqref{param-0:1}--\eqref{param-0:5} in
that lemma.
%Lemma~\ref{Lem:param-0}.
Now, fixing $T > 0$, by variation-of-constants formula we have
\begin{align}
\notag
    \norm{\ue}{L^r}
    &\le \nrm{e^{t\Delta} \uie}{L^r}
    + \int_0^t \nrm{e^{(t-\sigma)\Delta} \nabla \cdot (
      \ue(\cdot, \sigma) f(|\nabla \ve(\cdot, \sigma)|^2)
      \nabla \ve(\cdot, \sigma))}{L^r}\,\intd{\sigma}
\\
    &=: I_{1,\ep}(t) + I_{2,\ep}(t)
%      \quad\mbox{for all}\ t \in (0, T)\ \mbox{and}\ \ep \in(0,1).
    \label{ueLr-0:ueLr}
\end{align}
for all $t\in(0,T)$ and $\ep\in(0,1)$.
In view of \eqref{Mass_App},
an application of Lemma~\ref{Lem:Semi}~\eqref{Semi-1}
in conjunction with \eqref{uie:mass} yields $c_1 > 0$ such that
\begin{align}
\notag
    I_{1,\ep}(t)
    &\le \nrm{e^{t\Delta}(\uie - \cl{\uie})}{L^r} + \nrm{\cl{\uie}}{L^r}
\\ \notag
    &\le c_1(1 + t^{-\frac{n}{2}(1 - \frac{1}{r})}) \nrm{\uie - \cl{\uie}}{L^1}
    + \frac{m}{|\Omega|^{1-\frac{1}{r}}}
\\
    &\le 2 c_1 m (1 + t^{-\frac{n}{2}(1 - \frac{1}{r})})
    + \frac{m}{|\Omega|^{1-\frac{1}{r}}}
%      \quad\mbox{for all}\ t \in (0, T)\ \mbox{and}\ \ep \in (0,1).
    \label{ueLr-0:I1ep}
\end{align}
for all $t\in (0,T)$ and $\ep\in (0,1)$.
To estimate $I_{2,\ep}(t)$,
we invoke Lemma~\ref{Lem:Semi}~\eqref{Semi-3} along with \eqref{Ass:f:Main}
to find $c_2 > 0$ fulfilling
\begin{align}\label{ueLr-0:I2ep}
    I_{2,\ep}(t)
    \le c_2 k_f \int_0^t (1 + (t-\sigma)^{
      -\frac{1}{2}-\frac{n}{2}(\frac{1}{s}-\frac{1}{r})})
      \nrm{\ue(\cdot, \sigma) |\nabla\ve(\cdot, \sigma)|^{1-2\alpha}}{L^s}
      \,\intd{\sigma}
\end{align}
for all $t\in(0,T)$ and $\ep\in(0,1)$.
In order to further estimate herein,
we abbreviate
\begin{align}\label{ueLr-0:Rep}
    R_{\ep}(T) := \sup_{\sigma \in (0,T)}
      \sigma^{\frac{n}{2}(1-\frac{1}{r})} \nrm{\ue(\cdot,\sigma)}{L^r}
        \quad\mbox{for}\ \ep\in (0,1),
\end{align}
and use the H\"{o}lder inequality as well as the interpolation inequality
to obtain
\begin{align}
\notag
    &\nrm{\ue(\cdot, \sigma) |\nabla\ve(\cdot, \sigma)|^{1-2\alpha}}{L^s}
\\ \notag
    &\quad\,\le \nrm{\ue(\cdot, \sigma)}{L^{s_1}}
      \nrm{\nabla\ve(\cdot, \sigma)}{L^{s_{2}(1-2\alpha)}}^{1-2\alpha}
\\ \notag
    &\quad\,\le \nrm{\ue(\cdot, \sigma)}{L^1}^{1-\theta}
      \nrm{\ue(\cdot, \sigma)}{L^r}^{\theta}
      \nrm{\nabla\ve(\cdot, \sigma)}{L^{s_{2}(1-2\alpha)}}^{1-2\alpha}
\\ \notag
    &\quad\,= m^{1-\theta}
      (\sigma^{\frac{n}{2}(1-\frac{1}{r})}\nrm{\ue(\cdot,\sigma)}{L^r})^{\theta}
      \sigma^{-\frac{n}{2}(1-\frac{1}{r})\theta}
      \nrm{\nabla\ve(\cdot, \sigma)}{L^{s_{2}(1-2\alpha)}}^{1-2\alpha}
\\
    &\quad\,\le m^{1-\theta} R_{\ep}^{\theta}(T)
      \sigma^{-\frac{n}{2}(1-\frac{1}{r})\theta}
      \nrm{\nabla\ve(\cdot, \sigma)}{L^{s_{2}(1-2\alpha)}}^{1-2\alpha}
    \label{ueLr-0:intp}
\end{align}
for any $\sigma \in (0,T)$ and $\ep \in (0,1)$,
where $\theta := \frac{s_1 - 1}{(1-\frac{1}{r})s_1} \in (0,1)$.
According to Lemma~\ref{Lem:vexLp-0},
the property \eqref{param-0:3} in Lemma~\ref{Lem:param-0}
ensures that we can choose $c_3 > 0$ such that
\[
    \nrm{\nabla \ve(\cdot, \sigma)}{L^{s_{2}(1-2\alpha)}}^{1-2\alpha}
    \le c_3
\]
for any $\sigma \in (0,T)$ and $\ep\in (0,1)$.
Combining this with \eqref{ueLr-0:intp},
we thus can find a constant $c_4 = c_4(c_2, c_3, k_f, m, \theta) > 0$ such that
\eqref{ueLr-0:I2ep} turns into
\begin{align}
\notag
    I_{2,\ep}(t)
    &\le c_4 R_{\ep}^{\theta}(T) \int_0^t (1 + (t-\sigma)^{
      -\frac{1}{2}-\frac{n}{2}(\frac{1}{s}-\frac{1}{r})})
      \sigma^{-\frac{n}{2}(1-\frac{1}{r})\theta}\,\intd{\sigma}
\\
    &= c_4 R_{\ep}^{\theta}(T) \int_0^t (1 + (t-\sigma)^{
      -\frac{1}{2}-\frac{n}{2}(\frac{1}{s}-\frac{1}{r})})
      \sigma^{-\frac{n}{2}\big(1-\frac{1}{s_1}\big)}\,\intd{\sigma}
      \label{ueLr-0:I2ep2}
\end{align}
for every $t \in (0, T)$ and $\ep \in (0,1)$.
We then have to compute the integral in the right of this.
Noting that $\frac{1}{2} + \frac{n}{2}(\frac{1}{s} - \frac{1}{r}) < 1$ and
$\frac{n}{2}(1 - \frac{1}{s_1}) < 1$ from \eqref{param-0:1} and \eqref{param-0:5}
in Lemma~\ref{Lem:param-0},
a straightforward calculation yields
$c_5 = c_5(n, r, s, s_1) > 0$ fulfilling
\begin{align*}
    &\int_0^t (1 + (t-\sigma)^{
      -\frac{1}{2}-\frac{n}{2}(\frac{1}{s}-\frac{1}{r})})
      \sigma^{-\frac{n}{2}\big(1-\frac{1}{s_1}\big)}\,\intd{\sigma}
\\
    &\quad\,= \int_0^{\frac{t}{2}} (1 + (t-\sigma)^{
      -\frac{1}{2}-\frac{n}{2}(\frac{1}{s}-\frac{1}{r})})
      \sigma^{-\frac{n}{2}\big(1-\frac{1}{s_1}\big)}\,\intd{\sigma}
      + \int_{\frac{t}{2}}^t (1 + (t-\sigma)^{
      -\frac{1}{2}-\frac{n}{2}(\frac{1}{s}-\frac{1}{r})})
      \sigma^{-\frac{n}{2}\big(1-\frac{1}{s_1}\big)}\,\intd{\sigma}
\\
    &\quad\,\le \left(1 + \left(\frac{t}{2}\right)^{\!
      -\frac{1}{2}-\frac{n}{2}(\frac{1}{s}-\frac{1}{r})}\right)
    \int_{0}^{\frac{t}{2}} \sigma^{
      -\frac{n}{2}\big(1-\frac{1}{s_1}\big)}\,\intd{\sigma}
\\
    &\qquad\, + \left(\frac{t}{2}\right)^{\!-\frac{n}{2}\big(1-\frac{1}{s_1}\big)}\!
      \int_{\frac{t}{2}}^t (1 + (t-\sigma)^{
      -\frac{1}{2}-\frac{n}{2}(\frac{1}{s}-\frac{1}{r})})\,\intd{\sigma}
\\
    &\quad\,= \left(1 + \left(\frac{t}{2}\right)^{\!
      -\frac{1}{2}-\frac{n}{2}(\frac{1}{s}-\frac{1}{r})}\right)
      \frac{1}{1 - \frac{n}{2}(1 - \frac{1}{s_1})}
      \left(\frac{t}{2}\right)^{\!1 - \frac{n}{2}\big(1 - \frac{1}{s_1}\big)}
\\
    &\qquad\, + \left(\frac{t}{2}\right)^{\!-\frac{n}{2}\big(1-\frac{1}{s_1}\big)}
      \left(\frac{t}{2} + \frac{1}{\frac{1}{2} - \frac{n}{2}(\frac{1}{s}-\frac{1}{r})}
      \left(\frac{t}{2}\right)^{\!
        \frac{1}{2} - \frac{n}{2}(\frac{1}{s}-\frac{1}{r})}\right)
\\
    &\quad\, \le c_5 t^{1-\frac{n}{2}\big(1-\frac{1}{s_1}\big)}
      + c_5 t^{\frac{1}{2} - \frac{n}{2}(\frac{1}{s}-\frac{1}{r})
        - \frac{n}{2}\big(1 - \frac{1}{s_1}\big)}
\end{align*}
for all $t \in (0, T)$.
Inserting this into \eqref{ueLr-0:I2ep2} shows that
\begin{align}\label{ueLr-0:I2ep3}
    I_{2,\ep}(t)
    \le c_4 c_5 R_{\ep}^{\theta}(T)
      \big(t^{1-\frac{n}{2}\big(1-\frac{1}{s_1}\big)}
        + t^{\frac{1}{2} - \frac{n}{2}(\frac{1}{s}-\frac{1}{r})
        - \frac{n}{2}\big(1 - \frac{1}{s_1}\big)}\big)
\end{align}
for all $t \in (0, T)$ and $\ep \in (0, 1)$.
Collecting \eqref{ueLr-0:ueLr}, \eqref{ueLr-0:I1ep}
and \eqref{ueLr-0:I2ep3},
we thus infer the existence of
$c_6 = c_6(c_1, r, |\Omega|, m, c_4, c_5) > 0$ such that
\[
    \norm{\ue}{L^r}
    \le c_6(1 + t^{-\frac{n}{2}(1 - \frac{1}{r})})
      + c_6 R_{\ep}^{\theta}(T)
      \big(t^{1-\frac{n}{2}\big(1-\frac{1}{s_1}\big)}
        + t^{\frac{1}{2} - \frac{n}{2}(\frac{1}{s}-\frac{1}{r})
        - \frac{n}{2}\big(1 - \frac{1}{s_1}\big)}\big)
\]
for any $t \in (0, T)$ and $\ep \in (0, 1)$.
Multiplying this by $t^{\frac{n}{2}(1 - \frac{1}{r})}$ entails that
\begin{align}
\notag
    t^{\frac{n}{2}(1 - \frac{1}{r})}\norm{\ue}{L^r}
    &\le c_6 (t^{\frac{n}{2}(1 - \frac{1}{r})} + 1)
    + c_6 R_{\ep}^{\theta}(T)
    \big(t^{1+\frac{n}{2}\big(\frac{1}{s_1} - \frac{1}{r}\big)}
        + t^{\frac{1}{2} - \frac{n}{2}\frac{1}{s_2}}\big)
\\
    &\le c_6 (T^{\frac{n}{2}(1 - \frac{1}{r})} + 1)
    + c_6 R_{\ep}^{\theta}(T)
    \big(T^{1+\frac{n}{2}\big(\frac{1}{s_1} - \frac{1}{r}\big)}
        + T^{\frac{1}{2} - \frac{n}{2}\frac{1}{s_2}}\big)
    \label{ueLr-0:t_ueLr}
\end{align}
for any $t \in (0, T)$ and $\ep \in (0,1)$, where we have used that
\[
    \frac{1}{2} - \frac{n}{2}\left(\frac{1}{s} - \frac{1}{r}\right)
      - \frac{n}{2}\left(1 - \frac{1}{s_1}\right)
      + \frac{n}{2}\left(1 - \frac{1}{r}\right)
      = \frac{1}{2} - \frac{n}{2}\left(\frac{1}{s} - \frac{1}{s_1}\right)
%\\
    = \frac{1}{2} - \frac{n}{2} \cdot \frac{1}{s_2},
\]
valid by Lemma~\ref{Lem:param-0}~\eqref{param-0:4},
and that
\begin{align*}
    \frac{1}{2} - \frac{n}{2}\cdot \frac{1}{s_2}
    &= \frac{1}{2} - \frac{n}{2}\left(\frac{1}{s} - \frac{1}{s_1}\right)
\\
    &> \frac{1}{2} - \frac{n}{2}\left(\frac{1}{s} - \frac{1}{r}\right)
\\
    &> \frac{1}{2} - \frac{n}{2} \cdot \frac{1}{n}
\\
    &= 0
\end{align*}
by Lemma~\ref{Lem:param-0}~\eqref{param-0:1}, \eqref{param-0:2}.
%Therefore, \eqref{ueLr-0:t_ueLr} implies that for some
From \eqref{ueLr-0:t_ueLr} we therefore obtain
$c_7 = c_7(c_6, T, n, r, s_1, s_2) > 0$ such that
%we obtain
%
\[
    t^{\frac{n}{2}(1 - \frac{1}{r})}\norm{\ue}{L^r}
    \le c_7 + c_7 R_{\ep}^{\theta}(T)
\]
for all $t \in (0, T)$ and $\ep \in (0, 1)$.
Since the right-hand side of this is independent of $t\in (0, T)$,
it holds that
\[
    R_{\ep}(T) \le c_7 + c_7 R_{\ep}^{\theta}(T)
\]
for any $\ep \in (0,1)$.
Now, by virtue of the Young inequality, we obtain that
\[
    (1 - \theta) R_{\ep}(T) \le c_7 
    + (1 - \theta) c_7^{\frac{1}{1 - \theta}}
\]
for all $\ep \in (0,1)$, and thus
\[
    R_{\ep}(T) \le \frac{c_7}{1 - \theta} + c_7^{\frac{1}{1 - \theta}} =: c_8
\]
for any $\ep \in (0,1)$,
which along with \eqref{ueLr-0:Rep} leads to
\[
    t^{\frac{n}{2}(1 - \frac{1}{r})}\norm{\ue}{L^r}
    \le c_8
\]
for all $t \in (0, T)$ and $\ep \in (0,1)$.
Since $c_8 > 0$ is independent of $\ep \in (0,1)$, this finally
yields \eqref{ueLr-0:Ct}.
%with $\gamma = r$.
Furthermore, the inequality \eqref{ueLr-0:IntCt} results from integrating
\eqref{ueLr-0:Ct} on $(0,T)$.
\end{proof}
%-----------------------------------------------------------------%
%
%%%%%%%%%%%%%%%%%%%%%%%%%%%%%%%%%%%%%%%%%%%%%%%%%%%%%%%%%%%%%%%%%%
%                                                 Remark 3.3                                             %
%%%%%%%%%%%%%%%%%%%%%%%%%%%%%%%%%%%%%%%%%%%%%%%%%%%%%%%%%%%%%%%%%%
%
\begin{remark}\label{Rem:gamma}
We see from the proof of Lemma~\ref{Lem:ueLr-0} that
if $r$ satisfies \eqref{param-0:r},
then we can take $\gamma = r$ in \eqref{ueLr-0:Ct} and \eqref{ueLr-0:IntCt}.
\end{remark}

In preparation of our analysis regarding regularity of solutions
away from $t = 0$,
we close this section with uniform boundedness of $\ve$ in
$L^{\infty}(0, \infty; L^p(\Omega))$ for some $p > 0$,
which is a consequence of Lemma~\ref{Lem:vexLp-0}.
%
%%%%%%%%%%%%%%%%%%%%%%%%%%%%%%%%%%%%%%%%%%%%%%%%%%%%%%%%%%%%%%%%%%
%                                                 Lemma 3.5                                             %
%%%%%%%%%%%%%%%%%%%%%%%%%%%%%%%%%%%%%%%%%%%%%%%%%%%%%%%%%%%%%%%%%%
%
\begin{lem}\label{Lem:veLp-0}
Let $n \ge 2$ and let $p \in (0, q_1]$,
where $q_1 > q$ satisfies $\frac{1}{q_1} = \frac{1}{q}-\frac{1}{n}$.
Then there exists $C = C(p, q, q_1, m, n) > 0$ such that
\[
    \norm{\ve}{L^p} \le C
      \quad\mbox{for all}\ t > 0\ \mbox{and}\ \ep\in (0,1).
\]
\end{lem}
%%%%%%%%%%%%%%%%%%%%%%%%%%%%%%%%%%%%%%%%%%%%%%%%%%%%%%%%%%%%%%%%%%
%------------------------------proof------------------------------%
\begin{proof}
Lemma~\ref{Lem:vexLp-0} provides us with
$c_1 = c_1(|\Omega|, q, m, n) > 0$ such that
\[
    \norm{\nabla\ve}{L^q} \le c_1
\]
for all $t > 0$ and $\ep \in (0,1)$.
We thereupon infer from the H\"{o}lder inequality
and the embedding $W^{1,q}(\Omega) \embd L^{q_1}(\Omega)$ that
with $c_2 > 0$ we have
\begin{align*}
    \norm{\ve}{L^p}
    &\le |\Omega|^{\frac{1}{p} - \frac{1}{q_1}} \norm{\ve}{L^{q_1}}
\\
    &\le c_2 |\Omega|^{\frac{1}{p} - \frac{1}{q_1}} \norm{\nabla\ve}{L^q}
\\
    &\le c_1 c_2 |\Omega|^{\frac{1}{p} - \frac{1}{q_1}}
\end{align*}
for any $t > 0$ and $\ep\in (0,1)$, which concludes the proof.
\end{proof}
%-----------------------------------------------------------------%

%%==============================================================%%
%%==============================================================%%
%%                                               Section 4                                              %%
%%                               Constructing solution candidates                         %%
%%==============================================================%%
%%==============================================================%%
\section{Constructing solution candidates}\label{Sec:Inside}

This section is devoted to construct a pair of functions $(u,v)$
which solves the
boundary value problem in the system \eqref{Sys:Main} classically and satisfies
the regularity \eqref{Result:reg}.
%This will be achieved in Lemma~\ref{
This will be achieved on the basis of obtaining
sufficient H\"{o}lder bounds for $\ue$ and $\ve$ in
Lemmas~\ref{Lem:veHol2} and \ref{Lem:ueHol2}.
%To do this, 
The most important ingredient for its derivation is uniform boundedness
of $\norm{\ue}{L^{\infty}}$ on time intervals which is away from $t = 0$.

We begin with estimating the function
$t\mapsto \nabla \ve(\cdot, t)$ in $L^p(\Omega)$
on some time intervals for some $p \ge 1$,
which can be selected larger than the one in Lemma~\ref{Lem:vexLp-0}.
This would be possible with the aid of
estimates for $\ue$ and $\ve$ obtained in
Lemmas~\ref{Lem:ueLr-0} and \ref{Lem:veLp-0}.
%
%%%%%%%%%%%%%%%%%%%%%%%%%%%%%%%%%%%%%%%%%%%%%%%%%%%%%%%%%%%%%%%%%%
%                                                 Lemma 4.1                                             %
%%%%%%%%%%%%%%%%%%%%%%%%%%%%%%%%%%%%%%%%%%%%%%%%%%%%%%%%%%%%%%%%%%
%
\begin{lem}\label{Lem:vexLp}
Let $n \ge 2$ and $\alpha \in (\frac{n-2}{2(n-1)}, \frac{1}{2})$.
Suppose that $p \in (1, q_1]$,
where $q_1 > q$ is as in Lemma~\ref{Lem:veLp-0}.
Then for all $T > 0$ and $\tau \in (0, T)$,
there exists $C = C(\tau, T) > 0$ such that
\begin{align}\label{vexLp:vexLp}
    \norm{\nabla\ve}{L^p} \le C
      \quad\mbox{for all}\ t \in (\tau, T)\ \mbox{and}\ \ep \in (0,1).
\end{align}
%
%for all $t \in (\tau, T)$ and $\ep \in (0,1)$.
\end{lem}
%%%%%%%%%%%%%%%%%%%%%%%%%%%%%%%%%%%%%%%%%%%%%%%%%%%%%%%%%%%%%%%%%%
%------------------------------proof------------------------------%
\begin{proof}
Letting $r = p$, we first note that this $r > 1$ satisfies the condition
\eqref{ueLr-0:r}.
Indeed, for $n \ge 3$, according to \eqref{Ass:q} we observe that
\[
    \frac{1}{q_1} = \frac{1}{q} - \frac{1}{n}
%\\
    > \frac{n-1}{n} - \frac{1}{n}
%\\
    = \frac{n-2}{n},
\]
whence we have $r = p \le q_1 < \frac{n}{n-2}$.
Now, fixing $T > 0$ and $\tau \in (0, T)$,
Lemma~\ref{Lem:ueLr-0} warrants that with some
$c_1 = c_1(\tau, T) > 0$ we have
\begin{align}\label{vexLp:ueLp}
    \norm{\ue}{L^p} \le c_1
\end{align}
for all $t \in (\frac{\tau}{2}, T)$ and $\ep \in (0, 1)$.
Moreover, we see that Lemma~\ref{Lem:veLp-0} applies so as to
yield $c_2 = c_2(\tau, T) > 0$ satisfying
\begin{align}\label{vexLp:veLp}
    \nrm{\ve(\cdot, \tfrac{\tau}{2})}{L^p} \le c_2
\end{align}
for all $\ep \in (0,1)$.
Then, by Lemma~\ref{Lem:Semi}~\eqref{Semi-2} with
\eqref{vexLp:ueLp} and \eqref{vexLp:veLp}, we obtain $c_3 > 0$ such that
\begin{align*}
    &\norm{\nabla \ve}{L^p}
\\
    &\quad\,\le e^{-(t - \frac{\tau}{2})}
      \nrm{\nabla e^{(t - \frac{\tau}{2})\Delta} \ve(\cdot, \tfrac{\tau}{2})}{L^p}
    + \int_{\frac{\tau}{2}}^{t} e^{-(t-\sigma)}
      \left\|\nabla e^{(t-\sigma)\Delta}
        \frac{\ue(\cdot, \sigma)}{1 + \ep\ue(\cdot, \sigma)}\right\|_{
          L^p(\Omega)} \, \intd{\sigma}
\\
    &\quad\,\le c_3 e^{-(t-\frac{\tau}{2})} 
      \big(1 + (t-\tfrac{\tau}{2})^{-\frac{1}{2}}\big)
      e^{-\lambda_1 (t-\frac{\tau}{2})}
      \nrm{\ve(\cdot, \tfrac{\tau}{2})}{L^p}
\\
    &\qquad\,+ c_3 \int_{\frac{\tau}{2}}^{t} e^{-(t-\sigma)}
       \big(1 + (t-\sigma)^{-\frac{1}{2}}\big) e^{-\lambda_1 (t-\sigma)}
       \nrm{\ue(\cdot, \sigma)}{L^p} \,\intd{\sigma}
\\
    &\quad\,\le c_2 c_3 \big(1 + (t-\tfrac{\tau}{2})^{-\frac{1}{2}}\big)
      + c_1 c_3
        \int_{0}^{t - \frac{\tau}{2}} (1 + z^{-\frac{1}{2}}) e^{-(1+\lambda_1)z}
          \,\intd{z}
\\
    &\quad\,\le c_2 c_3 \big(1 + (\tfrac{\tau}{2})^{-\frac{1}{2}}\big)
      + c_1 c_3
        \int_{0}^{\infty} (1 + z^{-\frac{1}{2}}) e^{-(1+\lambda_1)z}
          \,\intd{z}
\end{align*}
for any $t\in (\tau, T)$ and $\ep\in (0,1)$.
%where $\lambda_1 > 0$ is as in Lemma~\ref{Lem:Semi}.
Since the integral
$\int_{0}^{\infty} (1 + z^{-\frac{1}{2}}) e^{-(1+\lambda_1)z} \,\intd{z}$
is finite, this shows \eqref{vexLp:vexLp}.
\end{proof}
%-----------------------------------------------------------------%

In light of the outcome of Lemma~\ref{Lem:vexLp}
we can now once more argue similar as in
Lemmas~\ref{Lem:param-0} and \ref{Lem:ueLr-0} to
derive $L^r$-estimates for $\ue$ with some $r > 1$,
which is possible to choose larger than in \eqref{ueLr-0:r}.
%
%%%%%%%%%%%%%%%%%%%%%%%%%%%%%%%%%%%%%%%%%%%%%%%%%%%%%%%%%%%%%%%%%%
%                                                 Lemma 4.2                                             %
%%%%%%%%%%%%%%%%%%%%%%%%%%%%%%%%%%%%%%%%%%%%%%%%%%%%%%%%%%%%%%%%%%
%
\begin{lem}\label{Lem:ueLr}
Let $n \ge 2$ and $\alpha \in (\frac{n-2}{2(n-1)}, \frac{1}{2})$.
Suppose that $r$ satisfies
\begin{align}\label{ueLr:r}
    \begin{cases}
      r \in (1, \infty] & \mbox{if}\ n\in\{2,3\}, \\
      r \in (1, \infty) & \mbox{if}\ n = 4, \\
      r \in (1, \frac{n}{n-4}) & \mbox{if}\ n \ge 5.
    \end{cases}
\end{align}
Then for any $T > 0$ and $\tau \in (0,T)$,
there exists $C = C(\tau, T) > 0$ such that
\[
    \norm{\ue}{L^r} \le C
      \quad\mbox{for all}\ t \in (\tau, T)\ \mbox{and}\ \ep\in (0,1).
\]
\end{lem}
%%%%%%%%%%%%%%%%%%%%%%%%%%%%%%%%%%%%%%%%%%%%%%%%%%%%%%%%%%%%%%%%%%
%------------------------------proof------------------------------%
\begin{proof}
Based on Lemma~\ref{Lem:ueLr-0}, we only need to consider the case that
$r$ satisfies
\[
%\begin{align}\label{ueLr:rS}
    \begin{cases}
      r = \infty & \mbox{if}\ n = 2, \\
      r \in [\frac{n}{n-2}, \infty] & \mbox{if}\ n = 3, \\
      r \in [\frac{n}{n-2}, \infty) & \mbox{if}\ n = 4, \\
      r \in [\frac{n}{n-2}, \frac{n}{n-4}) & \mbox{if}\ n \ge 5.
    \end{cases}
%\end{align}
\]
We proceed similarly to the proof of Lemma~\ref{Lem:param-0},
and obtaining
$s, s_1, s_2 > 1$ as well as $\gamma$ such that
$\gamma$ satisfies \eqref{ueLr-0:r}, that
$\frac{1}{\gamma} - \frac{1}{r} < \frac{2}{n}$, and that
\begin{multicols}{2}
\raggedcolumns
    \begin{enumerate}%[{\rm (i)}]
        \setlength{\leftskip}{5mm}
    \item[{\rm (i)}] $\frac{1}{s}-\frac{1}{r} < \frac{1}{n}$, 
    \item[{\rm (iii)}] $s_2(1 - 2\alpha) \in (1, q_1]$,
    \item[{\rm (v)}] $\frac{1}{\gamma} - \frac{1}{s_1} < \frac{2}{n}$.
    \columnbreak
    \item[{\rm (ii)}] $s_1 \in (\min\{s, \gamma\}, r)$,
    \item[{\rm (iv)}] $\frac{1}{s} = \frac{1}{s_1} + \frac{1}{s_2}$,
    \end{enumerate}
\end{multicols}%
\noindent
Now, for any $T > 0$ and $\tau \in (0, T)$,
we deduce from the first equation of \eqref{Sys:Reg} that
\begin{align}
\notag
    &\norm{\ue}{L^r}
\\ \notag
    &\quad\,\le \nrm{e^{(t - \frac{\tau}{2}\Delta)} \ue(\cdot, \tfrac{\tau}{2})}{L^r}
    + \int_{\frac{\tau}{2}}^{t} \nrm{e^{(t-\sigma)\Delta} \nabla \cdot (
      \ue(\cdot, \sigma) f(|\nabla \ve(\cdot, \sigma)|^2)
      \nabla \ve(\cdot, \sigma))}{L^r}\,\intd{\sigma}
\\
    &\quad\, =: \til{I_{1,\ep}}(t) + \til{I_{2,\ep}}(t)
    \label{ueLr:ueLr}
\end{align}
for all $t \in (\tfrac{\tau}{2}, T)$ and $\ep \in (0,1)$.
Recalling Lemma~\ref{Lem:ueLr-0},
we can find $c_1 > 0$ fuifilling
%%
%\[
%    \norm{\ue}{L^{\gamma}} \le c_1
%\]
%
%for all $t \in [\frac{\tau}{2}, T)$ and $\ep \in (0,1)$,
$\nrm{\ue(\cdot, \frac{\tau}{2})}{L^{\gamma}} \le c_1$ for all $\ep \in (0,1)$,
so that
%we obtain that with some $c_2 > 0$,
estimating the term $\til{I_{1,\ep}}(t)$ similar as in \eqref{ueLr-0:I1ep},
we obtain that with some $c_2 = c_2(c_1, T)> 0$,
\begin{align}\label{ueLr:tilI1ep}
%\begin{align*}
    \til{I_{1,\ep}}(t) 
    \le c_2\big(1+(t-\tfrac{\tau}{2})^{
      -\frac{n}{2}(\frac{1}{\gamma} - \frac{1}{r})}\big)
%\\
%    &\le c_2\big(1+(\tfrac{\tau}{2})^{
%      -\frac{n}{2}(\frac{1}{\gamma} - \frac{1}{r})}\big) + c_2
%\end{align*}
\end{align}
for all $t \in (\frac{\tau}{2}, T)$ and $\ep\in (0,1)$.
As for the term $\til{I_{2,\ep}}(t)$,
in view of the fact that there is $c_3 = c_3(\tau, T) > 0$ such that
\[
    \norm{\nabla \ve}{L^{s_{2}(1-2\alpha)}}^{1-2\alpha} \le c_3
\]
for any $t \in (\frac{\tau}{2}, T)$ and $\ep \in (0,1)$
by Lemma~\ref{Lem:vexLp},
%similarly as for \eqref{ueLr-0:Rep}
we define
\[
    \til{R_{\ep}}(\tau, T) :=
      \sup_{\sigma \in (\frac{\tau}{2}, T)} (\sigma-\tfrac{\tau}{2})^{
      \frac{n}{2}(\frac{1}{\gamma} - \frac{1}{r})} \nrm{
          \ue(\cdot, \sigma)}{L^r}
          \quad\mbox{for}\ \ep\in (0,1),
\]
and estimate in the same way as in Lemma~\ref{Lem:ueLr-0}
to find some $c_4 = c_4(c_1, c_3, \tau, T)$ such that
\begin{align}\label{ueLr:tilI2ep}
    \til{I_{2.\ep}}(t) \le c_4 \left(\til{R_{\ep}}(\tau, T)\right)^{\eta}
      \left[(t - \tfrac{\tau}{2})^{1-\frac{n}{2}\big(\frac{1}{\gamma}-\frac{1}{s_1}\big)}
        + (t - \tfrac{\tau}{2})^{\frac{1}{2}-\frac{n}{2}(\frac{1}{s}-\frac{1}{r})
          -\frac{n}{2}\big(\frac{1}{\gamma}-\frac{1}{s_1}\big)}
      \right]
\end{align}
for all $t \in (\frac{\tau}{2}, T)$ and $\ep\in (0,1)$, where
$\eta := \frac{s_1 - \gamma}{(1 - \frac{\gamma}{r})s_1} \in (0,1)$.
Collecting \eqref{ueLr:ueLr}, \eqref{ueLr:tilI1ep} and \eqref{ueLr:tilI2ep},
we conclude that with some $c_5 = c_5(\tau, T) > 0$ we have
\begin{align*}
    \norm{\ue}{L^r}
%\\
    &\le c_5\big(1+(t-\tfrac{\tau}{2})^{
      -\frac{n}{2}(\frac{1}{\gamma} - \frac{1}{r})}\big)
\\
    &\quad\,  + c_5 \left(\til{R_{\ep}}(\tau, T)\right)^{\eta}
      \left[(t - \tfrac{\tau}{2})^{1-\frac{n}{2}\big(\frac{1}{\gamma}-\frac{1}{s_1}\big)}
        + (t - \tfrac{\tau}{2})^{\frac{1}{2}-\frac{n}{2}(\frac{1}{s}-\frac{1}{r})
          -\frac{n}{2}\big(\frac{1}{\gamma}-\frac{1}{s_1}\big)}\right]
\end{align*}
%
%Lemma~\ref{
for any $t \in (\frac{\tau}{2}, T)$ and $\ep \in (0,1)$.
Multiplying this by $(t - \frac{\tau}{2})^{\frac{n}{2}(\frac{1}{\gamma}-\frac{1}{r})}$
and again arguing similarly in Lemma~\ref{Lem:ueLr-0},
we finally observe that there exists $c_6 = c_6(c_5, \tau, T, \eta) > 0$
such that
\[
    (t - \tfrac{\tau}{2})^{\frac{n}{2}(\frac{1}{\gamma}-\frac{1}{r})}
    \norm{\ue}{L^r} \le c_6
\]
for all $t \in (\frac{\tau}{2}, T)$ and $\ep \in (0,1)$,
which yields the assertion.
\end{proof}
%-----------------------------------------------------------------%

Repeating the arguments from Lemmas~\ref{Lem:veLp-0},
\ref{Lem:vexLp} and \ref{Lem:ueLr},
we obtain $\ep$-independent bounds for $t\mapsto\norm{\ue}{L^{\infty}}$
in any dimensions.
%
%%%%%%%%%%%%%%%%%%%%%%%%%%%%%%%%%%%%%%%%%%%%%%%%%%%%%%%%%%%%%%%%%%
%                                                 Lemma 4.3                                             %
%%%%%%%%%%%%%%%%%%%%%%%%%%%%%%%%%%%%%%%%%%%%%%%%%%%%%%%%%%%%%%%%%%
%
\begin{lem}\label{Lem:ueLinfty}
Let $n\in\N$, and assume that $\alpha > 0$ fulfills \eqref{param-0:alpha}.
Then for any $T > 0$ and $\tau \in (0, T)$, there exists $C = C(n, \tau, T) > 0$
such that
\begin{align}\label{ueLinfty:ueLinfty}
    \norm{\ue}{L^{\infty}} \le C
      \quad\mbox{for all}\ t \in (\tau, T)\ \mbox{and}\ \ep\in(0,1).
\end{align}
\end{lem}
%%%%%%%%%%%%%%%%%%%%%%%%%%%%%%%%%%%%%%%%%%%%%%%%%%%%%%%%%%%%%%%%%%
%------------------------------proof------------------------------%
\begin{proof}
For $n \in \{1, 2, 3\}$, this is true by
Lemmas~\ref{Lem:ueLr-0} and \ref{Lem:ueLr},
%thus we only have to check when $n \ge 4$.
and thus it is sufficient to prove the case $n \ge 4$.
Once again going back to the argument in Lemma~\ref{Lem:veLp-0},
we see from Lemma~\ref{Lem:vexLp} that for $p \in (0, q_2]$, where
$q_2 > q_1$ satisfies
$\frac{1}{q_2} = \frac{1}{q} - \frac{2}{n}$,
and for $T > 0$ as well as $\tau \in (0, T)$, there exists
$c_1 = c_1(\tau, T) > 0$ such that
\[
    \norm{\ve}{L^p} \le c_1
      \quad\mbox{for all}\ t \in (\tau, T)\ \mbox{and}\ \ep\in(0,1).
\]
Then we repeat the process in Lemma~\ref{Lem:vexLp}, with utilizing this instead of
\eqref{vexLp:veLp} and applying Lemma~\ref{Lem:ueLr} to see that
for any $p \in (1, q_2]$, $T > 0$ and $\tau \in (0, T)$,
there is $c_2 = c_2(\tau, T) > 0$ such that
\[
    \norm{\nabla \ve}{L^p} \le c_2
      \quad\mbox{for all}\ t \in (\tau, T)\ \mbox{and}\ \ep \in (0,1).
\]
Finally, using this instead of Lemma~\ref{Lem:vexLp} and arguing
as in Lemma~\ref{Lem:ueLr},
we conclude that for any $r$ satisfying
\[%\label{ueLr:r}
    \begin{cases}
      r \in (1, \infty] & \mbox{if}\ n\in\{4,5\}, \\
      r \in (1, \infty) & \mbox{if}\ n = 6, \\
      r \in (1, \frac{n}{n-6}) & \mbox{if}\ n \ge 7,
    \end{cases}
\]
and for all $T > 0$ as well as $\tau \in (0,T)$,
there exists $c_3 = c_3(\tau, T) > 0$ such that
\[
    \norm{\ue}{L^r} \le c_3
      \quad\mbox{for all}\ t\in(\tau, T)\ \mbox{and}\ \ep\in (0,1).
\]
This in particular ensures that the claim also holds in the case $n \in \{4, 5\}$.
We may iterate these arguments so as to achieve \eqref{ueLinfty:ueLinfty}
for any $n\in \N$.
\end{proof}
%-----------------------------------------------------------------%

Now a key toward our construction of solution candidates
will consist in the derivation of H\"{o}lder estimates for $\ue$ and $\ve$.
We begin with H\"{o}lder continuity of $\ue$ and $\ve$,
which will be achieved by applying the standard parabolic regularity
in \cite{LSU-1968}.
%
%%%%%%%%%%%%%%%%%%%%%%%%%%%%%%%%%%%%%%%%%%%%%%%%%%%%%%%%%%%%%%%%%%
%                                                 Lemma 4.4                                             %
%%%%%%%%%%%%%%%%%%%%%%%%%%%%%%%%%%%%%%%%%%%%%%%%%%%%%%%%%%%%%%%%%%
%
\begin{lem}\label{Lem:ueveHolder}
Let $n\in \N$, and let $\alpha > 0$ satisfy \eqref{param-0:alpha}.
Then for any $T > 0$ and $\tau \in (0, T)$,
there exist $\beta \in (0,1)$ and $C = C(\tau, T) > 0$ such that
\begin{align}
    &\hnrm{\ue}{\beta}{\tau, T} \le C\label{ueveHolder:ueHolder}
\\ \intertext{and}
    &\hnrm{\ve}{\beta}{\tau, T} \le C\label{ueveHolder:veHolder}
\end{align}
for all $\ep \in (0,1)$.
\end{lem}
%%%%%%%%%%%%%%%%%%%%%%%%%%%%%%%%%%%%%%%%%%%%%%%%%%%%%%%%%%%%%%%%%%
%------------------------------proof------------------------------%
\begin{proof}
%We will apply \cite[Theorem \three.10.1]{LSU-1968} to prove
%\eqref{ueveHolder:ueHolder}.
%To do this, 
%We first prove \eqref{ueveHolder:ueHolder}.
In order to apply \cite[Theorem \three.10.1]{LSU-1968} for
\eqref{ueveHolder:ueHolder},
we show uniform boundedness of $\nabla \ve$ in
$L^{\infty}(\tau, T; L^{\infty}(\Omega))$ for any $T > 0$ and $\tau \in (0,T)$.
We note that Lemma~\ref{Lem:veLp-0} guarantees existence of $c_1 > 0$ such that
\begin{align}\label{ueveHolder:veL1}
    \norm{\ve}{L^1} \le c_1
\end{align}
for all $t > 0$ and $\ep \in (0,1)$.
Fixing $T > 0$ and $\tau \in (0, T)$,
we see from Lemma~\ref{Lem:ueLinfty} that with some $c_2 = c_2(\tau, T) > 0$,
\begin{align}\label{ueveHolder:ueLinfty}
    \norm{\ue}{L^{\infty}} \le c_2
\end{align}
for all $t \in (\frac{\tau}{2}, T)$ and $\ep \in (0,1)$.
Invoking Lemma~\ref{Lem:Semi}~\eqref{Semi-2} along with
\eqref{ueveHolder:veL1} and \eqref{ueveHolder:ueLinfty},
we find $c_3 > 0$ fulfilling
\begin{align}
\notag
    &\norm{\nabla \ve}{L^{\infty}}
\\ \notag
    &\quad\,\le e^{-(t - \frac{\tau}{2})}
    \nrm{\nabla e^{(t-\frac{\tau}{2})\Delta} \ve(\cdot, \tfrac{\tau}{2})}{L^{\infty}}
    + \int_{\frac{\tau}{2}}^{t} e^{-(t-\sigma)}
    \left\| \nabla e^{(t-\sigma)\Delta} \frac{\ue(\cdot,\sigma)}{
      1 + \ep\ue(\cdot, \sigma)}\right\|_{L^{\infty}(\Omega)}\,\intd{\sigma}
\\ \notag
    &\quad\,\le c_3 \big(1 + (t -\tfrac{\tau}{2})^{
      -\frac{1}{2} - \frac{n}{2}} \big) e^{-(1 + \lambda_1)(t - \frac{\tau}{2})}
      \nrm{\ve(\cdot, \tfrac{\tau}{2})}{L^1}
\\ \notag
    &\qquad\, + c_3 \int_{\frac{\tau}{2}}^{t} \big(1 + (t-\sigma)^{-\frac{1}{2}}\big)
      e^{-(1+\lambda_1)(t-\sigma)} \nrm{\ue(\cdot, \sigma)}{L^{\infty}}
        \,\intd{\sigma}
\\ \notag
    &\quad\, \le c_1 c_3 \big(1 + (t -\tfrac{\tau}{2})^{
      -\frac{1}{2} - \frac{n}{2}} \big)
      + c_2 c_3 \int_{\frac{\tau}{2}}^{t} \big(1 + (t-\sigma)^{-\frac{1}{2}}\big)
      e^{-(1+\lambda_1)(t-\sigma)} \,\intd{\sigma}
\\
    &\quad\, \le c_1 c_3
      \big(1 + (\tfrac{\tau}{2})^{-\frac{1}{2}-\frac{n}{2}}\big)
    + c_2 c_3 \int_{0}^{\infty} (1 + z^{-\frac{1}{2}})e^{-(1 + \lambda_1)z}
      \,\intd{z} =:c_4
      \label{ueveHolder:vexLinfty}
\end{align}
for any $t \in (\tau, T)$ and $\ep\in (0,1)$.
This proves that there are $\beta_1 \in (0,1)$ and a constant
$c_5 = c_5(\tau, T) > 0$ such that
\[
    \hnrm{\ue}{\beta_1}{\tau, T} \le c_5
\]
for any $\ep \in (0,1)$.
Now, according to the Poincar\'{e}--Wirtinger inequality,
we infer from \eqref{ueveHolder:veL1} and
\eqref{ueveHolder:vexLinfty} that
with some $c_6 = c_6(c_1, c_4, |\Omega|) > 0$ we have
\[
    \norm{\ve}{L^{\infty}} \le c_6
\]
for all $t \in (\tau, T)$ and $\ep\in (0,1)$.
Therefore, once more applying \cite[Theorem \three.10.1]{LSU-1968}, we obtain
$\beta_2 \in (0,1)$ and $c_7 = c_7(\tau, T) > 0$ such that
\[
    \hnrm{\ve}{\beta_2}{\tau, T} \le c_7
\]
for all $t \in (\tau, T)$ and $\ep \in (0,1)$,
which completes the proof of \eqref{ueveHolder:ueHolder}
and \eqref{ueveHolder:veHolder} with $\beta:=\min\{\beta_1, \beta_2\}$
and $C := \max\{c_5, c_7\}$.
\end{proof}
%-----------------------------------------------------------------%

With the aid of Lemma~\ref{Lem:ueveHolder},
we may draw on the classical parabolic Schauder theory in \cite{LSU-1968}
to obtain uniform bounds for $\ve$ in
certain $C^{2+\beta, 1 + \frac{\beta}{2}}$ spaces.
%
%%%%%%%%%%%%%%%%%%%%%%%%%%%%%%%%%%%%%%%%%%%%%%%%%%%%%%%%%%%%%%%%%%
%                                                 Lemma 4.5                                             %
%%%%%%%%%%%%%%%%%%%%%%%%%%%%%%%%%%%%%%%%%%%%%%%%%%%%%%%%%%%%%%%%%%
%
\begin{lem}\label{Lem:veHol2}
Let $n \in \N$, and suppose that $\alpha > 0$ fulfills \eqref{param-0:alpha}.
Then for all $T > 0$ and $\tau \in (0,T)$,
there are $\beta \in (0,1)$ and $C = C(\tau, T) > 0$ such that
\begin{align}\label{veHolder2:veSchauder}
    \hhnrm{\ve}{\beta}{\tau, T} \le C
      \quad\mbox{for all}\ \ep \in (0,1).
\end{align}
\end{lem}
%%%%%%%%%%%%%%%%%%%%%%%%%%%%%%%%%%%%%%%%%%%%%%%%%%%%%%%%%%%%%%%%%%
%------------------------------proof------------------------------%
\begin{proof}
For arbitrary $T > 0$ and $\tau \in (0, T)$, we define
\[
    \vtile(x,t) := \zeta(t)\ve(x,t)
      \quad\mbox{for}\ x \in \Ombar,\ t \in [\tfrac{\tau}{3}, T]
      \ \mbox{and}\ \ep \in (0,1),
\]
where $\zeta \in C^{\infty}((0,\infty))$ satisfies the conditions that
$0 \le \zeta \le 1$ on $(0, \infty)$ and
\[
    \zeta(t) = 
    \begin{cases}
      0 & \mbox{if}\ t \in [\tfrac{\tau}{3}, \tfrac{\tau}{2}], \\
      1 & \mbox{if}\ t \in [\tau, T].
    \end{cases}
\]
Abbreviating
\[
    g_{\ep}(x,t) := (\zeta'(t) - \zeta(t))\ve(x,t) + \zeta(t) 
        \frac{\ue(x,t)}{1 + \ep\ue(x,t)}
      \quad\mbox{for}\ x \in \Ombar,\ t \in [\tfrac{\tau}{3}, T]
      \ \mbox{and}\ \ep \in (0,1),
\]
we infer from \eqref{Sys:Reg} that $\vtile$ solves the problem
\begin{align*}%\label{Sys:Reg}
  \begin{cases}
    (\vtile)_t=\Delta \vtile + g_{\ep}
    & \mbox{in}\ \Omega \times (\tfrac{\tau}{3}, T),
  \\
    \nabla \vtile \cdot \nu = 0
    & \mbox{on}\ \pa\Omega \times (\tfrac{\tau}{3}, T),
  \\
    \vtile(\cdot , \tfrac{\tau}{3}) = 0
    & \mbox{in}\ \Omega.
  \end{cases}
\end{align*}
Since from Lemma~\ref{Lem:ueveHolder} we know that there are
$\beta \in (0,1)$ and $c_1 = c_1(\tau, T) > 0$ such that
\[
    \hnrm{g_{\ep}}{\beta}{\frac{\tau}{3}, T} \le c_1
\]
for any $\ep\in (0,1)$,
thanks to \cite[Theorem \four.5.3]{LSU-1968} we observe that
there exists a constant $c_2 = c_2(\tau, T) > 0$ fulfilling
\[
    \hhnrm{\vtile}{\beta}{\frac{\tau}{3}, T} \le c_2
\]
for all $\ep \in (0,1)$.
As $\zeta \equiv 1$ on $[\tau, T]$, we arrive at \eqref{veHolder2:veSchauder}.
\end{proof}
%-----------------------------------------------------------------%

In Lemma~\ref{Lem:ueHol2}, we will show that
$\ue$ belongs to $\phhsp{\beta}{\tau, T}$
for any $T > 0$ and $\tau \in (0, T)$ with some $\beta \in (0,1)$,
by using a similar agument as in Lemma~\ref{Lem:veHol2}
with the cutoff function $\zeta$.
However, different from the previous lemma,
the chemotactic term in \eqref{Sys:Reg} now requires
H\"{o}lder continuity of $\nabla \ue$ to apply the Schauder theory
in \cite{LSU-1968}.
This can be established with the aid of \cite[Theorem 1.1]{L-1987}.
%
%%%%%%%%%%%%%%%%%%%%%%%%%%%%%%%%%%%%%%%%%%%%%%%%%%%%%%%%%%%%%%%%%%
%                                                 Lemma 4.6                                             %
%%%%%%%%%%%%%%%%%%%%%%%%%%%%%%%%%%%%%%%%%%%%%%%%%%%%%%%%%%%%%%%%%%
%
\begin{lem}\label{Lem:uexHol}
Let $n\in \N$, and let $\alpha > 0$ satisfy \eqref{param-0:alpha}.
Then for any $T > 0$ and $\tau \in (0, T)$,
there exist $\beta \in (0,1)$ and $C = C(\tau, T) > 0$ such that
\[
    \htnrm{\ue}{\beta}{\tau, T} \le C
      \quad\mbox{for all}\ \ep \in (0,1).
\]
\end{lem}
%%%%%%%%%%%%%%%%%%%%%%%%%%%%%%%%%%%%%%%%%%%%%%%%%%%%%%%%%%%%%%%%%%
%------------------------------proof------------------------------%
\begin{proof}
Fixing $T > 0$ and $\tau \in (0, T)$, and letting
\[
    \utile(x,t) := \zeta(t)\ue(x,t)
      \quad\mbox{for}\ x \in \Ombar,\ t \in [\tfrac{\tau}{3}, T]
      \ \mbox{and}\ \ep \in (0,1)
\]
with the cutoff function $\zeta \in C^{\infty}((0,\infty))$ as in
Lemma~\ref{Lem:veHol2}, we see that $\utile$ is the unique solution of
the system
\begin{align}\label{uexHol:Sys}
  \begin{cases}
    (\utile)_t = \nabla \cdot A_{\ep} + B_{\ep}
    & \mbox{in}\ \Omega \times (\tfrac{\tau}{3}, T),
  \\
    \nabla \utile \cdot \nu = 0
    & \mbox{on}\ \pa\Omega \times (\tfrac{\tau}{3}, T),
  \\
    \utile(\cdot , \tfrac{\tau}{3}) = 0
    & \mbox{in}\ \Omega,
  \end{cases}
\end{align}
where
\begin{align*}
    &A_{\ep}(x,t) := \nabla\utile(x,t) - \utile(x,t)
      f(|\nabla\ve(x,t)|^2)\nabla \ve(x,t)
%      \quad\mbox{for}\ x \in \Ombar,\ t \in [\tfrac{\tau}{3}, T]
%      \ \mbox{and}\ \ep \in (0,1)
\\ \intertext{and}
    &B_{\ep}(x,t) := \zeta'(t) \ue(x,t)
\end{align*}
for $x \in \Ombar$, $t \in [\frac{\tau}{3}, T]$ and $\ep \in (0,1)$.
Noting that $f\in C^2([0, \infty))$ and $f \le k_f$ on $[0, \infty)$,
by virtue of Lemmas~\ref{Lem:ueveHolder} and \ref{Lem:veHol2},
we deduce that the function $A_{\ep}$ and $B_{\ep}$ satisfy the conditions
\cite[(1.2a)--(1.2c), (1.3), (1.5) and (1.6)]{L-1987}
with suitable constants which are independent of $\ep \in (0,1)$.
Therefore, by \cite[Theorem 1.1]{L-1987},
we can find $\beta\in (0,1)$ and $c_1 = c_1(\tau, T) > 0$ such that
\[
    \htnrm{\utile}{\beta}{\frac{\tau}{3}, T} \le c_1
%      \quad\mbox{for all}\ \ep \in (0,1).
\]
for all $\ep \in (0,1)$, completing the proof of this lemma.
\end{proof}
%-----------------------------------------------------------------%

%
%%%%%%%%%%%%%%%%%%%%%%%%%%%%%%%%%%%%%%%%%%%%%%%%%%%%%%%%%%%%%%%%%%
%                                                 Remark 4.1                                            %
%%%%%%%%%%%%%%%%%%%%%%%%%%%%%%%%%%%%%%%%%%%%%%%%%%%%%%%%%%%%%%%%%%
%
\begin{remark}
The above lemma is the first place in which we use boundedness
of the function $f$ on some interval including $0$.
This implies that we can also establish H\"{o}lder bounds for solutions
as in Lemmas~\ref{Lem:ueveHolder} and \ref{Lem:veHol2}
even when the function $f$ is not bounded near $0$
(see for instance \cite[Lemmas 4.1--4.4]{JRT-2023} and \cite[Lemma 3.4]{Ko1}).
\end{remark}

Lemma~\ref{Lem:uexHol} now enables us to gain uniform bounds for $\ue$
%$\hhnrm{\ue}{\beta}{\tau, T}$ uniformly
in $\phhsp{\beta}{\tau, T}$
for any $T > 0$ and $\tau \in (0, T)$
with some $\beta\in (0,1)$.
%
%%%%%%%%%%%%%%%%%%%%%%%%%%%%%%%%%%%%%%%%%%%%%%%%%%%%%%%%%%%%%%%%%%
%                                                 Lemma 4.7                                             %
%%%%%%%%%%%%%%%%%%%%%%%%%%%%%%%%%%%%%%%%%%%%%%%%%%%%%%%%%%%%%%%%%%
%
\begin{lem}\label{Lem:ueHol2}
Let $n \in \N$, and suppose that $\alpha > 0$ fulfills \eqref{param-0:alpha}.
Then for all $T > 0$ and $\tau \in (0,T)$,
there are $\beta \in (0,1)$ and $C = C(\tau, T) > 0$ such that
%
%\[
\begin{align}\label{ueHol2:ueSchauder}
    \hhnrm{\ue}{\beta}{\tau, T} \le C
      \quad\mbox{for all}\ \ep \in (0,1).
\end{align}
%\]
%
\end{lem}
%%%%%%%%%%%%%%%%%%%%%%%%%%%%%%%%%%%%%%%%%%%%%%%%%%%%%%%%%%%%%%%%%%
%------------------------------proof------------------------------%
\begin{proof}
Fixing $T > 0$ and $\tau \in (0,T)$ arbitrary, we rewrite the first equation in
\eqref{uexHol:Sys} as
\[
%\begin{align*}%\label{uexHol:Sys}
%  \begin{cases}
    (\utile)_t = \Delta \utile - h_{\ep} + B_{\ep},
%    & \mbox{in}\ \Omega \times (\tfrac{\tau}{3}, T),
%  \\
%    \nabla \utile \cdot \nu = 0
%    & \mbox{on}\ \pa\Omega \times (\tfrac{\tau}{3}, T),
%  \\
%    \utile(\cdot , \tfrac{\tau}{3}) = 0
%    & \mbox{in}\ \Omega,
%  \end{cases}
%\end{align*}
\]
where
\begin{align*}
    &h_{\ep}(x,t)
\\
    &\quad\, := \zeta(t) \left[f(|\nabla\ve(x,t)|^2)\nabla \ue(x,t)\cdot\nabla \ve(x,t)
    + \ue(x,t) \nabla \cdot \big(
      f(|\nabla\ve(x,t)|^2)\nabla \ve(x,t)\big)\right]
\end{align*}
for $x \in \Ombar$, $t \in [\frac{\tau}{3}, T]$ and $\ep \in (0,1)$.
Then we employ Lemmas~\ref{Lem:ueveHolder}, \ref{Lem:veHol2}
and \ref{Lem:uexHol} to see that with some
$\beta \in (0,1)$ and $c_1 = c_1(\tau, T) > 0$,
%as well as $c_2 = c_2(\tau, T) > 0$,
%
\[
    \hnrm{h_{\ep}}{\beta}{\tau, T} \le c_1
      \quad\mbox{and}\quad
    \hnrm{B_{\ep}}{\beta}{\tau, T} \le c_1
\]
for all $\ep \in (0,1)$.
Thus, similarly to Lemma~\ref{Lem:veHol2} we once again apply
\cite[Theorem \four.5.3]{LSU-1968} to conclude the proof.
\end{proof}
%-----------------------------------------------------------------%

Having thus asserted \eqref{veHolder2:veSchauder} and
\eqref{ueHol2:ueSchauder},
we are now in a position to construct a pair of nonnegative functions
$(u,v)$ which solves the problem \eqref{Sys:Main} classically
with the regularity \eqref{Result:reg}.
%
%%%%%%%%%%%%%%%%%%%%%%%%%%%%%%%%%%%%%%%%%%%%%%%%%%%%%%%%%%%%%%%%%%
%                                                 Lemma 4.8                                             %
%%%%%%%%%%%%%%%%%%%%%%%%%%%%%%%%%%%%%%%%%%%%%%%%%%%%%%%%%%%%%%%%%%
%
\begin{lem}\label{Lem:uvC21}
Let $n\in \N$, and let $\alpha > 0$ satisfy \eqref{param-0:alpha}.
Then there exist $(\ep_j)_{j\in\N}\subset (0,1)$ as well as
nonnegative functions $u,v \in C^{\infty}(\Ombar \times (0, \infty))$
such that $\ep_j \dwto 0$ as $j\to\infty$ and
\begin{align}
    &\ue \to u
      \quad\mbox{in}\ C^{2,1}_{\rm loc}(\Ombar\times (0,\infty)),
      \label{uvC21:uC21}
\\
    &\ve \to v
      \quad\mbox{in}\ C^{2,1}_{\rm loc}(\Ombar\times (0,\infty))
      \label{uvC21:vC21}
\end{align}
as $\ep = \ep_j \dwto 0$.
Moreover, $(u,v)$ solves \eqref{Sys:Main}
in the classical sense in $\Ombar \times (0,\infty)$,
and satisfies the mass conservation property \eqref{Result:mass}.
\end{lem}
%%%%%%%%%%%%%%%%%%%%%%%%%%%%%%%%%%%%%%%%%%%%%%%%%%%%%%%%%%%%%%%%%%
%------------------------------proof------------------------------%
\begin{proof}
The statements \eqref{uvC21:uC21} and \eqref{uvC21:vC21}
along with a suitable sequence and functions
are direct consequences of Lemmas~\ref{Lem:veHol2} and \ref{Lem:ueHol2}.
From these, we observe that $(u,v)$ is a classical solution of \eqref{Sys:Main},
whereas \eqref{Result:mass} results from \eqref{Mass_App}.
\end{proof}
%-----------------------------------------------------------------%

%%==============================================================%%
%%==============================================================%%
%%                                               Section 5                                              %%
%%                          Continuity at t=0: Proof of Theorem 1.1                       %%
%%==============================================================%%
%%==============================================================%%
\section{Continuity at $t = 0$: Proof of Theorem~\ref{Thm:Main}}\label{Sec:t0}

The goal of this section is to show the properties
\eqref{Result:u} and \eqref{Result:v} for the functions $u,v$ from
Lemma~\ref{Lem:uvC21},
and to prove Theorem~\ref{Thm:Main}.
As a preparation to the proof of \eqref{Result:u},
we begin with showing the integral of $\ue|\nabla\ve|^{1-2\alpha}$
on $\Omega\times (0,T)$ tends to 0 when $T \dwto 0$.
The idea of the proof is based on \cite{H-2023P}.
%
%%%%%%%%%%%%%%%%%%%%%%%%%%%%%%%%%%%%%%%%%%%%%%%%%%%%%%%%%%%%%%%%%%
%                                                 Lemma 5.1                                             %
%%%%%%%%%%%%%%%%%%%%%%%%%%%%%%%%%%%%%%%%%%%%%%%%%%%%%%%%%%%%%%%%%%
%
\begin{lem}\label{Lem:taxisL1}
Let $\alpha > 0$ satisfy \eqref{param-0:alpha}, and suppose that
$r$ satisfies \eqref{param-0:r}.
Then for all $T > 0$, there exists $C > 0$ such that
\[
    \int_0^t \nrm{\ue(\cdot, \sigma) |\nabla \ve(\cdot, \sigma)|^{1-2\alpha}}
    {L^1} \,\intd{\sigma}
    \le C t^{1-\frac{n}{2}(1-\frac{1}{r})}
      \quad\mbox{for all}\ t \in (0, T)\ \mbox{and}\ \ep\in (0,1).
\]
\end{lem}
%%%%%%%%%%%%%%%%%%%%%%%%%%%%%%%%%%%%%%%%%%%%%%%%%%%%%%%%%%%%%%%%%%
%------------------------------proof------------------------------%
\begin{proof}
We fix $T > 0$ arbitrary.
%As in the proof of
%Since $r$ satisfies \eqref{param-0:r},
Noting that \eqref{param-0:rela} ensures
\[
    r_{\alpha} := \frac{1 - 2\alpha}{1 - \frac{1}{r}} < q,
\]
we see from
Lemma~\ref{Lem:vexLp-0} that
there is $c_1 > 0$ satisfying
\[
    \norm{\nabla \ve}{L^{r_{\alpha}}}^{1-2\alpha} \le c_1
\]
for any $t > 0$ and $\ep \in (0,1)$,
and hence by the H\"{o}lder inequality,
\begin{align}
\notag
    \nrm{\ue(\cdot, t) |\nabla \ve(\cdot, t)|^{1-2\alpha}}{L^1}
    &\le \norm{\ue}{L^r}\norm{\nabla \ve}{L^{r_{\alpha}}}^{1-2\alpha}
\\
    &\le c_1 \norm{\ue}{L^r}
    \label{taxisL1:Holder}
\end{align}
for all $t > 0$ and $\ep \in (0,1)$.
Moreover, in light of Remark~\ref{Rem:gamma},
we employ Lemma~\ref{Lem:ueLr-0} to find $c_2 > 0$ such that
\[
    \int_{0}^{t} \nrm{\ue(\cdot, \sigma)}{L^r} \,\intd{\sigma}
    \le c_2 t^{1 - \frac{n}{2}(1 - \frac{1}{r})}
\]
for all $t \in (0,T)$ and $\ep \in (0,1)$.
This together with \eqref{taxisL1:Holder} implies that
\begin{align*}
    \int_0^t \nrm{\ue(\cdot, \sigma) |\nabla \ve(\cdot, \sigma)|^{1-2\alpha}}
    {L^1} \,\intd{\sigma}
    &\le c_1 \int_{0}^{t} \nrm{\ue(\cdot, \sigma)}{L^r} \,\intd{\sigma}
\\
    &\le c_1 c_2 t^{1 - \frac{n}{2}(1 - \frac{1}{r})}
\end{align*}
for all $t \in (0, T)$ and $\ep \in (0,1)$.
Therefore we complete the proof of Lemma~\ref{Lem:taxisL1}.
\end{proof}
%-----------------------------------------------------------------%

As a consequence of Lemma~\ref{Lem:taxisL1},
we can show that $\ue(\cdot, t) - \uie$ tends uniformly to $0$
as $t \dwto 0$ in the vague topology on $\Radon$.
This will be an important ingredient for proving the property \eqref{Result:u}
in Lemma~\ref{Lem:Limit_u}.
%
%%%%%%%%%%%%%%%%%%%%%%%%%%%%%%%%%%%%%%%%%%%%%%%%%%%%%%%%%%%%%%%%%%
%                                                 Lemma 5.2                                             %
%%%%%%%%%%%%%%%%%%%%%%%%%%%%%%%%%%%%%%%%%%%%%%%%%%%%%%%%%%%%%%%%%%
%
\begin{lem}\label{Lem:ueuie}
Suppose that $\alpha > 0$ satisfies \eqref{param-0:alpha}.
Then for any $\varphi \in C^0(\Ombar)$ and $\delta > 0$,
there is $t_0 = t_0(\varphi, \delta) > 0$ such that
\[
    \left|\io \ue(\cdot, t)\varphi - \io \uie\varphi\right| \le \delta
      \quad\mbox{for all}\ t \in (0, t_0)\ \mbox{and}\ \ep\in (0,1).
\]
\end{lem}
%%%%%%%%%%%%%%%%%%%%%%%%%%%%%%%%%%%%%%%%%%%%%%%%%%%%%%%%%%%%%%%%%%
%------------------------------proof------------------------------%
\begin{proof}
We first note that the set
\[
    X := \{\varphi \in C^2(\Ombar) \mid \varphi \cdot \nu = 0
    \ \mbox{on}\ \pa\Omega\}
\]
is dense in $C^0(\Ombar)$.
%(see e.g.\ \cite[Lemma 3.6]{LW-2023}
Indeed, this can be seen by adapting the argument
from \cite[Lemma 3.6]{LW-2023}.
From this, we fix $\varphi \in C^0(\Ombar)$ and $\delta > 0$,
then we find $\psi \in X$ such that
\begin{align}\label{ueuie:4mdelta}
    \nrm{\varphi - \psi}{L^{\infty}} \le \frac{\delta}{4m}.
\end{align}
%
%Also, we let $r$ satisfy \eqref{param-0:r}
We claim that there exists $t_0 \in (0,1)$ such that
\begin{align}\label{ueuie:claim}
    \left|\io \ue(\cdot, t)\psi - \io \uie\psi\right| \le \frac{\delta}{2}
\end{align}
for all $t \in (0, t_0)$ and $\ep \in (0,1)$.
To see this, we let $r$ satisfy \eqref{param-0:r},
then there is $c_1 > 0$ such that
\[
    \int_0^t \nrm{\ue(\cdot, \sigma) |\nabla \ve(\cdot, \sigma)|^{1-2\alpha}}
    {L^1} \,\intd{\sigma}
    \le c_1 t^{1-\frac{n}{2}(1-\frac{1}{r})}
\]
for all $t \in (0,1)$ and $\ep \in (0,1)$ by Lemma~\ref{Lem:taxisL1},
whence the first equation in \eqref{Sys:Reg}
along with \eqref{Ass:f:Main}
and \eqref{Mass_App} implies that
\begin{align*}
    &\left|\io \ue(\cdot, t)\psi - \io \uie\psi\right|
\\
    &\quad\,= \left|\int_0^t \io (\ue)_t \psi\right|
\\
    &\quad\,\le \left|\int_0^t \io \ue \Delta \psi\right|
    + \left|\int_0^t \io \ue f(|\nabla \ve|^2) \nabla \ve \cdot \nabla \psi\right|
\\
    &\quad\,\le \nrm{\Delta\psi}{L^{\infty}}
      \int_0^t \nrm{\ue(\cdot, \sigma)}{L^1} \, \intd{\sigma}
    + k_f \nrm{\nabla \psi}{L^{\infty}}
    \int_0^t \nrm{\ue(\cdot, \sigma) |\nabla \ve(\cdot, \sigma)|^{1-2\alpha}}
    {L^1} \,\intd{\sigma}
\\
    &\quad\,\le m\nrm{\Delta\psi}{L^{\infty}} t
      + c_1 k_f \nrm{\nabla \psi}{L^{\infty}}t^{1-\frac{n}{2}(1-\frac{1}{r})}
\end{align*}
for any $t \in (0,1)$ and $\ep \in (0,1)$.
Since the right-hand side of this inequality is independent of
$\ep \in (0,1)$ and tends to $0$ as $t\dwto 0$,
this establishes the claim.
Now for $t \in (0, t_0)$ and $\ep \in (0,1)$, we have
\begin{align}
\notag
        &\left|\io \ue(\cdot, t)\varphi - \io \uie\varphi\right|
\\
    &\quad\,\le
      \left|\io \ue(\cdot, t)\varphi - \io \ue(\cdot, t)\psi\right|
      + \left|\io \ue(\cdot, t)\psi - \io \uie\psi\right|
      + \left|\io \uie\psi - \io \uie\varphi\right|.
\label{ueuie:split}
\end{align}
Since according to \eqref{uie:mass}, \eqref{Mass_App} and \eqref{ueuie:4mdelta}
we know that
\[
    \left|\io \ue(\cdot, t)\varphi - \io \ue(\cdot, t)\psi\right|
    + \left|\io \uie\psi - \io \uie\varphi\right|
    \le \frac{\delta}{2}
\]
for all $t \in (0, t_0)$ and $\ep \in (0,1)$.
In light of \eqref{ueuie:split},
in conjunction with \eqref{ueuie:claim} this completes the proof.
%
%Therefore, in view of \eqref{uie:mass} and \eqref{Mass_App},
%thanks to \eqref{ueuie:4mdelta} as well as \eqref{ueuie:claim}
%it follows that
%we deduce from \eqref{ueuie:claim} and \eqref{ueuie:4mdelta} that
%
%\begin{align*}
%    &\left|\io \ue(\cdot, t)\varphi - \io \uie\varphi\right|
%\\
%    &\quad\,\le
%      \left|\io \ue(\cdot, t)\varphi - \io \ue(\cdot, t)\psi\right|
%      + \left|\io \ue(\cdot, t)\psi - \io \uie\psi\right|
%      + \left|\io \uie\psi - \io \uie\varphi\right|
%\\
%    &\quad\,\le 2m \nrm{\varphi - \psi}{L^{\infty}} + \frac{\delta}{2}
%\\
%    &\quad\,\le \delta
%\end{align*}
%
%for all $t \in (0, t_0)$ and $\ep \in (0,1)$.
\end{proof}
%-----------------------------------------------------------------%

Having at hand the important information gained in Lemma~\ref{Lem:ueuie},
we are now in a position to establish the property \eqref{Result:u}
for the function $u$ from Lemma~\ref{Lem:uvC21}.
%
%%%%%%%%%%%%%%%%%%%%%%%%%%%%%%%%%%%%%%%%%%%%%%%%%%%%%%%%%%%%%%%%%%
%                                                 Lemma 5.3                                             %
%%%%%%%%%%%%%%%%%%%%%%%%%%%%%%%%%%%%%%%%%%%%%%%%%%%%%%%%%%%%%%%%%%
%
\begin{lem}\label{Lem:Limit_u}
The function $u$ constructed in Lemma~\ref{Lem:uvC21}
has the property that
\[
    u(\cdot, t) \wsc \mu_0
    \quad\mbox{in}\ \Radon\quad \mbox{as}\ t \dwto 0.
\]
\end{lem}
%%%%%%%%%%%%%%%%%%%%%%%%%%%%%%%%%%%%%%%%%%%%%%%%%%%%%%%%%%%%%%%%%%
%------------------------------proof------------------------------%
\begin{proof}
For arbitrary fixed $\varphi \in C^0(\Ombar)$ and $\delta > 0$,
Lemma~\ref{Lem:ueuie} provides $t_0 = t_0(\varphi, \delta) > 0$
fulfilling
\begin{align}\label{Limit_u:ueuie}
    \left|\io \ue(\cdot, t)\varphi - \io \uie\varphi\right| \le \frac{\delta}{3}
\end{align}
for all $t\in (0, t_0)$ and $\ep \in (0,1)$,
whereas \eqref{uvC21:uC21} and \eqref{uie:Radon} guarantee that
for any $t \in (0, t_0)$, there exists $\ep(t) \in (0,1)$ such that
\[
    \left|\io u(\cdot, t)\varphi - \io u_{\ep(t)}(\cdot, t)\varphi\right|
    \le \frac{\delta}{3}
      \quad\mbox{and}\quad
    \left|\io u_{0\ep(t)}\varphi - \mu_0(\varphi)\right|
    \le \frac{\delta}{3}.
\]
Together with \eqref{Limit_u:ueuie} this proves
\[
    \left|\io u(\cdot, t) - \mu_0(\varphi)\right| \le \delta
\]
for all $t \in (0, t_0)$, which concludes the proof.
%, this completes the proof.
\end{proof}
%-----------------------------------------------------------------%

The only remaining step then is to show continuity of $v$ at $t = 0$
in the sense of \eqref{Result:v}.
To this end, we will follow an approach inspired by a strategy
already pursed in \cite[Lemma 5.4]{H-2023P},
namely estimating the norm
$\nrm{\ve(\cdot, t) - \vie}{W^{1,q}}$
based on the semigroup estimates from Lemma~\ref{Lem:Semi}.
%together with the estimates for $\ue(\cdot, t)$ from $t = 0$,
%which is obtained in Lemma~\ref{Lem:ueLr-0}.
%The idea of the proof is taken from \cite[Lemma 5.4]{H-2023P}.
%pursued in \cite[Lemma 5.4]{H-2023P}
%
%%%%%%%%%%%%%%%%%%%%%%%%%%%%%%%%%%%%%%%%%%%%%%%%%%%%%%%%%%%%%%%%%%
%                                                 Lemma 5.4                                             %
%%%%%%%%%%%%%%%%%%%%%%%%%%%%%%%%%%%%%%%%%%%%%%%%%%%%%%%%%%%%%%%%%%
%
\begin{lem}\label{Lem:Limit_v}
Let the function $v$ be as in Lemma~\ref{Lem:uvC21}.
Then
\[
    v(\cdot, t) \to v_0
    \quad\mbox{in}\ W^{1,q}(\Omega)\quad \mbox{as}\ t\dwto 0.
\]
\end{lem}
%%%%%%%%%%%%%%%%%%%%%%%%%%%%%%%%%%%%%%%%%%%%%%%%%%%%%%%%%%%%%%%%%%
%------------------------------proof------------------------------%
\begin{proof}
%This can be shown similarly to \cite[Lemma 5.4]{H-2023P}.
Fixing $\delta > 0$, we know from the second equation in \eqref{Sys:Reg} that
\begin{align}
\notag
    \nrm{\ve(\cdot, t) - \vie}{W^{1,q}}
    &\le \nrm{e^{t(\Delta - 1)}\vie - \vie}{W^{1,q}}
    + \int_0^t \left\|
      e^{(t-\sigma)(\Delta - 1)} \frac{\ue(\cdot,\sigma)}
      {1 + \ep\ue(\cdot,\sigma)}\right\|_{W^{1,q}(\Omega)}\,\intd{\sigma}
\\
    &=: J_{1,\ep}(t) + J_{2,\ep}(t)
    \label{Limit_v:J1epJ2ep}
\end{align}
for all $t > 0$ and $\ep \in (0,1)$,
whence the definition
%the definition of $\vie$
\eqref{Def:vie} of $\vie$ together with
%In light of \eqref{Def:vie},
%due to
the standard properties of the Neumann heat semigroup
as in \cite[Lemma 5.4]{H-2023P}
provides $t_1 > 0$ such that
%we can fix $t_1 > 0$ such that
%
%\[
\begin{align}\label{Limit_v:J1ep}
    J_{1,\ep}(t)
    = \nrm{e^{\ep(\Delta-1)}(e^{t(\Delta-1)}v_0 - v_0)}{W^{1,q}}
    \le \frac{\delta}{4}
\end{align}
%\]
%
for all $t \in (0, t_1)$ and $\ep \in (0,1)$.
On the other hand,
in light of \eqref{Mass_App}
we employ Lemma~\ref{Lem:Semi}~\eqref{Semi-1}
and \eqref{Semi-2} to derive that
with some $c_1 > 0$,
\begin{align}
\notag
    &J_{2,\ep}(t)
\\ \notag
    &\quad\,= \int_0^t \left\|
      e^{(t-\sigma)(\Delta - 1)} \frac{\ue(\cdot,\sigma)}
      {1 + \ep\ue(\cdot,\sigma)}\right\|_{L^{q}(\Omega)}\,\intd{\sigma}
      + \int_0^t \left\|\nabla
      e^{(t-\sigma)(\Delta - 1)} \frac{\ue(\cdot,\sigma)}
      {1 + \ep\ue(\cdot,\sigma)}\right\|_{L^{q}(\Omega)}\,\intd{\sigma}
\\ \notag
    &\quad\,\le c_1 m \int_0^t \big(1 + (t-\sigma)^{-\frac{n}{2}\big(1 - \frac{1}{q}\big)}\big)
      %e^{-(1+\lambda_1) (t - \sigma)}
      \,\intd{\sigma}
    + \frac{m}{|\Omega|^{1-\frac{1}{q}}} t
%\\
    %&\quad\, 
    + c_1 m \int_0^t \big(1 + (t-\sigma)^{-\frac{1}{2}-\frac{n}{2}\big(
      1 - \frac{1}{q}\big)}\big) 
      %e^{-(1+\lambda_1)(t-\sigma)}
      \, \intd{\sigma}
\\
    &\quad\,= \left(2 c_1 m + \frac{m}{|\Omega|^{1-\frac{1}{q}}}\right)t
    + \frac{c_1 m}{1 - \frac{n}{2}\big(1 - \frac{1}{q}\big)}
      t^{1-\frac{n}{2}\big(1 - \frac{1}{q}\big)}
    + \frac{c_1m}{\frac{1}{2} - \frac{n}{2}\big(1 - \frac{1}{q}\big)}
      t^{\frac{1}{2} - \frac{n}{2}\big(1 - \frac{1}{q}\big)}
\label{Limit_v:J2right}
\end{align}
for all $t > 0$ and $\ep \in (0,1)$.
According to \eqref{Ass:q}, we see that
$\frac{1}{2} - \frac{n}{2}(1 - \frac{1}{q}) > 0$ holds
and hence the right-hand side of \eqref{Limit_v:J2right}
%noting that the right-hand side of the above inequality
tends to $0$ as $t \dwto 0$.
We thereby obtain $t_2 > 0$ such that
%because
%$\frac{1}{2} - \frac{n}{2}(1 - \frac{1}{q}) > 0$ by \eqref{Ass:q},
%and thus we obtain $t_2 > 0$ such that
%
\begin{align}\label{Limit_v:J2ep}
    J_{2,\ep}(t) \le \frac{\delta}{4}
\end{align}
for any $t \in (0, t_2)$ and $\ep \in (0,1)$.
Hence, inserting \eqref{Limit_v:J1ep} and \eqref{Limit_v:J2ep} into
\eqref{Limit_v:J1epJ2ep}, we arrive at
\begin{align}\label{Limit_v:vevie}
    \nrm{\ve(\cdot, t) - \vie}{W^{1,q}} \le \frac{\delta}{2}
\end{align}
for all $t \in (0, t_0)$ and $\ep \in (0,1)$, where
$t_0 := \min\{t_1, t_2\}$.
Now we infer from \eqref{uvC21:vC21} and \eqref{Conv_vie} that
for each $t \in (0, t_0)$ there is $\ep(t) \in (0,1)$ satisfying
\[
    \nrm{v(\cdot, t) - v_{\ep(t)}(\cdot, t)}{W^{1,q}}
    \le \frac{\delta}{4}
      \quad\mbox{and}\quad
    \nrm{v_{0\ep(t)} - v_0}{W^{1,q}}
    \le \frac{\delta}{4}.
\]
Finally, combining this with \eqref{Limit_v:vevie} readily implies
\[
    \nrm{v(\cdot, t) - v_0}{W^{1,q}} \le \delta
\]
for all $t \in (0, t_0)$, which yields the claim.
%On the other hand, since $q$ satisfies \eqref{Ass:q},
%there is $\xi > 0$ satisfying
%
%\[
%    \frac{n}{2}\left(1 - \frac{1}{q}\right) + \xi < \frac{1}{2}.
%\]
%
%Then if we let $s := \frac{1}{2} + \xi$,
%the continuous embedding
%$D((-\Delta + 1)^s) \embd W^{1,q}(\Omega)$
%from \cite[Theorem 1.6.1]{H-1981}
%along with Lemma~\ref{Lem:FracS}
%enables us to find $c_1 > 0$ such that
%
%\begin{align}\label{Limit_v:J2ep}
%    J_{2,\ep}(t)
%    \le c_1 \int_0^t (t-\sigma)^{-s} \nrm{\ue(\cdot, \sigma)}{L^q}
%    \,\intd{\sigma}
%\end{align}
%
%for all $t \in (0, t_1)$ and $\ep \in (0,1)$.
%Noticing that $q$ satisfies the condition \eqref{ueLr-0:r},
%according to Lemma~\ref{Lem:ueLr-0}
%there are $q_{*} \in (1, \infty]$ with $1 - \frac{1}{q_{*}} < \frac{2}{n}$ and
%$c_2 > 0$ such that
%
%\[
%    \norm{\ue}{L^q} \le c_2 t^{-\frac{n}{2}\big(1 - \frac{1}{q_{*}}\big)}
%\]
%
%for all $t \in (0, t_1)$ and $\ep \in (0,1)$.
%Inserting this into \eqref{Limit_v:J2ep} yields
%
%\[
%    J_{2,\ep}(t) \le c_1 c_2 
%      \int_0^t (t-\sigma)^{-s} {\sigma}^{-\frac{n}{2}\big(1 - \frac{1}{q_{*}}\big)}
%      \,\intd{\sigma}
%\]
%
%for all $t \in (0, t_1)$ and $\ep \in (0,1)$.
%By straightforward computation,
\end{proof}
%-----------------------------------------------------------------%

Our main result has thereby been established already.

%------------------------------proof------------------------------%
\begin{prth1.1}
If $\alpha > 0$ satisfies \eqref{param-0:alpha},
%it is sufficient to combine Lemma~\ref{Lem:uvC21}
we take $u, v \in C^{\infty}(\Ombar \times (0, \infty))$ as given by
Lemma~\ref{Lem:uvC21} and then only need to combine
Lemma~\ref{Lem:Limit_u} with Lemma~\ref{Lem:Limit_v},
while in view of Remark~\ref{Rem:param-0:alpha},
it can be verified that the result continues to hold
%we can also prove our main result
for the case $\alpha \ge \frac{1}{2}$.
\qed
\end{prth1.1}

%
%%==============================================================%%
%%==============================================================%%
%%                                         Acknowledgments                                        %%
%%==============================================================%%
%%==============================================================%%
\smallskip
\section*{Acknowledgments}
The author would like to express gratitude to Professor Tomomi Yokota for
his warm encouragement and helpful comments on the manuscript.

%
%%==============================================================%%
%%==============================================================%%
%%                                              Reference                                              %%
%%==============================================================%%
%%==============================================================%%

\small\bibliographystyle{plain}

\end{document}